\documentclass[titlepage,12pt]{article} 
\usepackage{hyperref}
\usepackage[usenames,dvipsnames]{pstricks} 
\usepackage{pst-plot} 
\usepackage{amssymb,amsthm,amsmath} 
\usepackage[a4paper]{geometry}
\usepackage{datetime2}
\usepackage[utf8]{inputenc}
\usepackage[italian,english]{babel}
\usepackage{mathrsfs}
\usepackage{bbm}

\date{}

\newcommand{\R}{\mathbb{R}}

\newcommand{\re}{\mathbb{R}}

\newcommand{\n}{\mathbb{N}}

\newcommand{\ep}{\varepsilon}

\newcommand{\sPMF}{\operatorname{PMF}}
\newcommand{\PM}{\operatorname{\mathbb{PM}}}
\newcommand{\PMF}{\operatorname{\mathbb{PMF}}}
\newcommand{\RPM}{\operatorname{\mathbb{RPM}}}
\newcommand{\RPMF}{\operatorname{\mathbb{RPMF}}}
\newcommand{\J}{\operatorname{\mathbb{J}}}
\newcommand{\FF}{\operatorname{\mathbb{F}}}
\newcommand{\GG}{\operatorname{\mathbb{G}}}
\newcommand{\HH}{\operatorname{\mathbb{H}}}
\newcommand{\RPMV}{\operatorname{\mathbb{VRPM}}}
\newcommand{\RPMH}{\operatorname{\mathbb{HRPM}}}
\newcommand{\sgn}{\operatorname{sgn}}

\newcommand{\obl}{\mathrm{Obl}}
\newcommand{\vrt}{\mathrm{Vert}}
\newcommand{\orz}{\mathrm{Hor}}
\newcommand{\hatv}{\widehat{v}}
\newcommand{\hatw}{\widehat{w}}
\newcommand{\glim}{\Gamma\mbox{--}\lim}
\newcommand{\auto}{\mathrel{\substack{\vspace{-0.1ex}\\\displaystyle\leadsto\\[-0.9em] \displaystyle\leadsto}}}
\newcommand{\loc}{_{\mathrm{loc}}}
\newcommand{\omep}{\omega(\ep)}
\newcommand{\omepn}{\omega(\ep_{n})}

\newtheorem{thm}{Theorem}[section]
\newtheorem{thmbibl}{Theorem}[section]

\newtheorem{rmk}[thm]{Remark}
\newtheorem{prop}[thm]{Proposition}
\newtheorem{defn}[thm]{Definition}

\newtheorem{lemma}[thm]{Lemma}

\title{Multi-scale analysis of minimizers for a second order regularization of the Perona-Malik functional}

\author{
Massimo Gobbino\vspace{1ex}\\ 
{\normalsize Università di Pisa} \\
{\normalsize Dipartimento di Matematica}\\ 
{\normalsize PISA (Italy)}\\  
{\normalsize e-mail: \texttt{massimo.gobbino@unipi.it}}
\and
Nicola Picenni\vspace{1ex}\\ 
{\normalsize Università di Pisa} \\
{\normalsize Dipartimento di Matematica}\\ 
{\normalsize PISA (Italy)}\\
{\normalsize e-mail: \texttt{nicola.picenni@dm.unipi.it}}
}

\begin{document}
\maketitle

\begin{abstract}

We investigate the asymptotic behavior of minimizers for the singularly perturbed Perona-Malik functional in one dimension. In a previous study, we have shown that blow-ups of these minimizers at a suitable scale converge to staircase-like piecewise constant functions. 

Building upon these findings, we delve into finer scales, revealing that both the vertical and horizontal regions of the staircase steps display cubic polynomial behavior after appropriate rescaling.

Our analysis hinges on identifying the dominant terms of the functional within each regime, elucidating the mechanisms driving the observed asymptotic behavior.

\vspace{6ex}

\noindent{\bf Mathematics Subject Classification 2020 (MSC2020):} 
49J45, 35B25.

\vspace{6ex}



\end{abstract}


\section{Introduction}

In this paper, which is intended as a continuation of~\cite{FastPM-CdV}, we consider minimizers of the regularized Perona-Malik functional   
\begin{equation}
\sPMF_{\ep}(u):=
\int_{0}^{1}\left\{\ep^{10}|\log\ep|^{2}u''(x)^{2}+
\log\left(1+u'(x)^{2}\right)+
\beta(u(x)-f(x))^{2}\right\}dx,
\label{defn:SPM-intro}
\end{equation}
where $\ep\in(0,1)$ and $\beta>0$ are real numbers, and $f\in L^{2}((0,1))$ is a given function that we call \emph{forcing term}. The three quantities in the integral can be interpreted as follows.
\begin{itemize}

\item The last one is a \emph{fidelity term}, tuned by the parameter $\beta$, that penalizes the distance between $u$ and the forcing term $f$.

\item  The second term is the true Perona-Malik functional. Its main feature is that the Lagrangian $\phi(p):=\log(1+p^{2})$ is non-convex and has a convexification that is identically equal to zero. The convex/concave nature of the Lagrangian translates into the forward-backward character of the formal gradient-flow of this second part of the functional, which is (up to a factor~2) the celebrated evolution equation
\begin{equation}
u_{t}=\left(\frac{u_{x}}{1+u_{x}^{2}}\right)_{x}=\frac{1-u_{x}^{2}}{(1+u_{x}^{2})^{2}}\,u_{xx},
\label{defn:PM-eqn}
\end{equation}
introduced by P.~Perona and J.~Malik~\cite{PeronaMalik}. The contrast between the analytical instability that is expected for (\ref{defn:PM-eqn}), and the apparent stability of numerical simulations, is usually referred to as the ``Perona-Malik paradox'' (see~\cite{Kichenassamy,2001-CPAM-Esedoglu,2006-SIAM-BNP,GG:grad-est}). For a more comprehensive discussion of the preceding literature, we refer readers to the recent papers~\cite{2018-Nonlinearity-KimYan,2020-JDE-BerSmaTes,FastPM-CdV,2023-PM-TV,fastpm-discreto}, as well as the references cited therein.

\item  The first term is a second order singular perturbation that regularizes the functional, and corresponds in the dynamical setting to the fourth order regularization of (\ref{defn:PM-eqn}) (see \cite{1996-Duke-DeGiorgi,2006-DCDS-BelFusGug,2008-TAMS-BF,2019-SIAM-BerGiaTes}). The bizarre form of the $\ep$-dependent coefficient prevents the appearance of implicit rescaling factors and implicit decay rates in the sequel of the paper.

\end{itemize}

\paragraph{\textmd{\textit{First order analysis}}}

On the one hand, the convex and coercive second order term guarantees that the model is well-posed, in the sense that the minimum problem for (\ref{defn:SPM-intro}) admits at least one minimizer in $C^2([0,1])$ for every admissible choice of $\ep$, $\beta$, $f$. On the other hand, the unstable nature of the Perona-Malik functional becomes increasingly manifest as $\ep\to 0^+$. As a consequence, minimum values tend to~0, minimizers tend to $f$ in $L^2((0,1))$ and, more important, minimizers develop a microstructure known as \emph{staircasing effect}, as shown in the upper part of Figure~\ref{figure:multi-scale}.

In the companion paper~\cite{FastPM-CdV} we carried out a quantitative analysis of this phenomenon. We assumed that the forcing term is of class $C^1$, we introduced the quantity
\begin{equation}
    \omep:=\ep|\log\ep|^{1/2},
    \label{defn:omep}
\end{equation}
and we proved two results.
\begin{itemize}
    \item Minimum values satisfy the asymptotic estimate
    \begin{equation}
        \lim_{\ep\to 0^+}\frac{\min\{\sPMF_{\ep}(u):u\in C^2((0,1))\}}{\omep^2}=10\left(\frac{2\beta}{27}\right)^{1/5}\int_{0}^{1}|f'(x)|^{4/5}\,dx.
        \label{th:min-asympt}
    \end{equation}

    \item  Minimizers develop a staircase structure at scale $\omep$. In order to formalize this idea, for every family $\{u_\ep\}$ of minimizers of (\ref{defn:SPM-intro}), and for every family of points $\{x_\ep\}\subseteq (0,1)$ such that $x_\ep\to x_0 \in (0,1)$, we introduced the family of blow-ups
\begin{equation}
w_\ep(y):=\frac{u_{\ep}(x_{\ep}+\omep y)-f(x_{\ep})}{\omep}
\label{defn:wep-old}
\end{equation}
and the family of blow-ups
\begin{equation}
v_\ep(y):=\frac{u_{\ep}(x_{\ep}+\omep y)-u_{\ep}(x_{\ep})}{\omep}.
\label{defn:vep}
\end{equation}

In essence, both families entail a zooming-in process on the graph of $u_\ep$ within a window of size $\omep$. With (\ref{defn:wep-old}), the window's center lies on the graph of the forcing term, while with (\ref{defn:vep}), it resides on the graph of $u_\ep$ itself. We proved that both families converge (up to subsequences) to a piecewise constant function resembling a staircase, where the steps' height and length depend upon $f'(x_{0})$. The notion of convergence employed here is the strongest one that allows smooth functions to converge to discontinuous functions, specifically strict convergence of bounded variation functions (for further details, we refer to section~\ref{sec:previous}).
\end{itemize}

\paragraph{\textmd{\textit{Main contribution of this paper}}}

The aim of this paper is to describe the structure of minimizers at finer scales. To this end, we examine separately the vertical and the horizontal parts of each step.
\begin{itemize}

\item In the vertical region, we further rescale the blow-ups horizontally by a factor $\ep^2$, and we show that they converge in a very strong sense to a suitable cubic polynomial that interpolates the horizontal lines corresponding to the two consecutive steps (see Theorem~\ref{thm:vertical}). 

\item In the horizontal region, we further rescale the blow-ups vertically by a factor $\omep^2$, and we show that again they converge, albeit up to subsequences, to cubic polynomials (see Theorem~\ref{thm:horizontal}). This indicates that blow-ups exhibit flatness up to an order of $\omep^2$, and hence minimizers exhibit flatness up to an order of $\omep^3$, within these regions.
\end{itemize}

\begin{figure}[ht]
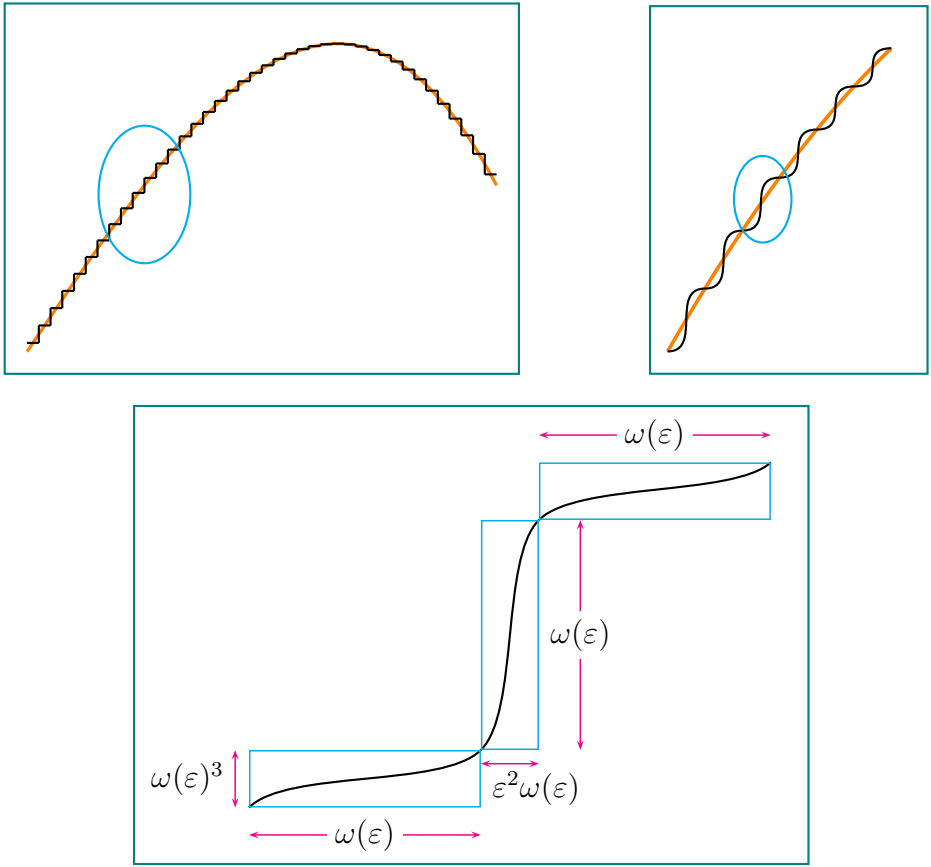

\centering

\hfill
\psset{unit=8.5ex}
\pspicture(-0.2,-0.2)(4.2,3)
\psframe[linecolor=teal](-0.2,-0.2)(4.2,3)

\psplot[linecolor=orange,linewidth=1.5\pslinewidth]{0}{4}{x 1.5 mul x 3 exp 14 div sub}

\multido{\n=0.05+0.10}{40}{
\psline(!\n \space 0.05 sub \n \space 1.5 mul \n \space 3 exp 14 div sub)
(!\n \space 0.05 add \n \space 1.5 mul \n \space 3 exp 14 div sub)
}

\multido{\n=0.05+0.10}{39}{
\psline
(!\n \space 0.05 add \n \space 1.5 mul \n \space 3 exp 14 div sub)
(!\n \space 0.05 add \n \space 0.1 add 1.5 mul \n \space 0.1 add 3 exp 14 div sub)
}

\psellipse[linecolor=cyan](1,1.35)(0.4,0.6)

\endpspicture
\hfill
\psset{unit=13.5ex}
\pspicture(-0.2,-0.3)(1.3,1.7)
\psframe[linecolor=teal](-0.2,-0.3)(1.3,1.7)

\psplot[linecolor=orange,linewidth=1.7\pslinewidth]{-0.1}{1.1}{x 1.7 mul x 2 exp 3 div sub}

\multido{\n=-0.1+0.2}{6}{\psbezier
(!\n \space \n \space 1.7 mul \n \space 2 exp 3 div sub)
(!\n \space 0.19 add \n \space 1.7 mul \n \space 2 exp 3 div sub)
(!\n \space 0.01 add \n \space 0.2 add 1.7 mul \n \space 0.2 add 2 exp 3 div sub)
(!\n \space 0.2 add \n \space 0.2 add 1.7 mul \n \space 0.2 add 2 exp 3 div sub)
}

\psellipse[linecolor=cyan](0.41,0.65)(0.16,0.24)

\endpspicture
\hfill
\mbox{}
\vspace{2ex}

\par
\hfill
\psset{unit=2.8ex}
\pspicture(-9,-3)(8.5,9)
\psframe[linecolor=teal](-9,-3)(8.5,9)

\psbezier(-6,-1.5)(-5,-0.5)(-1,-1)(0,0)
\psbezier(0,0)(1,1)(0.5,5)(1.5,6)
\psbezier(1.5,6)(2.5,7)(6.5,6.6)(7.5,7.5)

\psframe[linewidth=0.7\pslinewidth,linecolor=cyan](0,0)(1.5,6)
\psframe[linewidth=0.7\pslinewidth,linecolor=cyan](1.5,6)(7.5,7.5)
\psframe[linewidth=0.7\pslinewidth,linecolor=cyan](-6,-1.5)(0,0)

\pcline[offset=-1ex,linewidth=0.7\pslinewidth,linecolor=magenta]{<->}(0,0)(1.5,0)
\lput{0}{\rput[tl]{0}(-0.45,-0.2){$\ep^2\omep$}}

\pcline[offset=-2ex,linewidth=0.7\pslinewidth,linecolor=magenta]{<->}(-6,-1.5)(0,-1.5)
\lput*{:U}{$\omep$}

\pcline[offset=2ex,linewidth=0.7\pslinewidth,linecolor=magenta]{<->}(1.5,7.5)(7.5,7.5)
\lput*{:U}{$\omep$}

\pcline[offset=-3ex,linewidth=0.7\pslinewidth,linecolor=magenta]{<->}(1.5,0)(1.5,6)
\lput*{0}{$\omep$}

\pcline[offset=1ex,linewidth=0.7\pslinewidth,linecolor=magenta]{<->}(-6,-1.5)(-6,0)
\lput{0}{\rput[r]{0}(-0.3,0){$\omep^3$}}

\endpspicture
\hfill
\mbox{}

\caption{Representation of minimizers at three different levels of resolution.
\\ 
Top-left: staircasing effect around the forcing term (low resolution).
\\ 
Top-right: zoomed-in view of the staircase within a specific region of order $\omep$.
\\ 
Bottom: multi-scale analysis with cubic behavior of the staircase steps.}
\label{figure:multi-scale}
\end{figure}

These two statements offer a comprehensive understanding of a multi-scale problem.  In the original scale, minimizers are substantially indistinguishable from the forcing term.  Upon closer examination, they exhibit a staircase-like appearance, with steps of lengths and heights of the order of $\omep$. A deeper zoom reveals that the horizontal ``segments'' resemble cubic polynomials enclosed in boxes approximately $\omep\times\omep^3$ in size, while the vertical ``segments'' resemble cubic polynomials enclosed in boxes approximately $\ep^2\omep\times\omep$ in size.

These two statements are summed up in the bottom of Figure~\ref{figure:multi-scale}. 

\paragraph{\textmd{\textit{Brief heuristic explanation}}}

Despite the formal analogy, the two cubic behaviors have a completely different origin. Specifically, the cubic behavior in the vertical region stems from the singular perturbation term, while the cubic behavior in the horizontal region arises from a linearization of the remaining two terms.

As elucidated in~\cite{FastPM-CdV}, the fundamental cause of the staircasing phenomenon lies in the convex/concave nature of the logarithmic term, which forces minimizers to follow the forcing term by alternating ``almost horizontal regions'' where $u_\ep'(x)\sim 0$, and ``almost vertical regions'' where $u_\ep'(x)\sim\pm\infty$. Within each vertical region separating two consecutive horizontal regions, the objective is to minimize the contribution of the term involving second-order derivatives. Upon suitable rescaling, this results in a minimization problem of the form
\begin{equation}
    \min\left\{\int_{-1}^{1}u''(x)^2\,dx:
    u(-1)=-1,\ u(1)=1, \ u'(-1)=u'(1)=0\right\},
\nonumber
\end{equation}
whose solution is a cubic polynomial. 

On the contrary, in the horizontal regions we can neglect the second order derivatives, use the approximation $\log(1+u'(x)^2)\sim u'(x)^2$, and consider only the linear part $-2u(x)f(x)$ of the fidelity term (of course the term of order zero with $f(x)^2$ plays no role in the minimization process). At a suitable scale we can also replace $f(x)$ with the tangent line $f'(x_0)x$.  This yields a minimization problem of the form
\begin{equation}
    \min\left\{\int_{-1}^{1}\left[u'(x)^2-2\beta f'(x_0)x\cdot u(x)\right]dx\right\},
\nonumber
\end{equation}
whose solutions are again cubic polynomials, but for completely different reasons.

We conclude by noting that the first-order analysis, namely the staircase-like structure of minimizers, appears to be relatively independent of the approximating model (see the discussion in~\cite[section~7]{FastPM-CdV}, as well as the analogous results of~\cite{fastpm-discreto} for the discrete approximation). 
On the contrary, the higher-order analysis is likely to be significantly influenced by the specific approximating model, at least in the vertical regime. For instance, when considering a singular perturbation with derivatives of order $k$, as discussed in~\cite{Solci}, we anticipate that the vertical regions will resemble polynomials of degree $2k - 1$, whereas in the discrete approximation all vertical segments are realized by single big jumps. The dependency of the cubic behavior in the horizontal regime on the approximating model is currently less evident. Investigating the discrete approximation could shed light on this matter.

\paragraph{\textmd{\textit{Structure of the paper}}}

This paper is organized as follows. Section~\ref{sec:previous} is a summary of the notations and results of~\cite{FastPM-CdV} that are relevant for this paper. Section~\ref{sec:main} contains the statements of our main results. Section~\ref{sec:lemmata} is a collection of technical, but rather elementary, preliminaries. Section~\ref{sec:vertical} contains the proof of our main result concerning vertical parts of the staircase steps. Section~\ref{sec:horiz} contains the proof of our main result concerning horizontal parts of the staircase steps.  


\setcounter{equation}{0}
\section{Previous results -- First order analysis}\label{sec:previous}

In this section we recall the notation and the main results of~\cite{FastPM-CdV} that we need throughout this paper. 

\paragraph{\textmd{\textit{Functionals}}}

In order to emphasize the dependence on all the parameters, we write the functional (\ref{defn:SPM-intro}) in the more general form
\begin{equation}
\PMF_{\ep}(\beta,f,\Omega,u):=
\int_{\Omega}\left\{\ep^{6}\omep^{4}u''(x)^{2}+
\log\left(1+u'(x)^{2}\right)+
\beta(u(x)-f(x))^{2}\right\}dx,
\label{defn:SPM}
\end{equation}
where $\ep\in(0,1)$ and $\beta>0$ are real numbers, $\Omega\subseteq\re$ is an open set, and $f\in L^{2}(\Omega)$ is the forcing term. We consider also the functional
\begin{equation}\label{defn:RPM}
\RPM_\ep(\Omega,u):=
\int_{\Omega}\left\{\ep^{6}u''(x)^{2}+
\frac{1}{\omep^2} \log\left(1+u'(x)^{2}\right)\right\}dx,
\end{equation}
which is a rescaled version of the principal part of (\ref{defn:SPM}), and its variant with the fidelity term
\begin{equation}
\RPMF_\ep(\beta,f,\Omega,u):=\RPM_\ep(\Omega,u)+ \beta \int_\Omega (u(x)-f(x))^{2}\,dx.
\label{defn:RPMF}
\end{equation}

Finally, we consider the set of ``pure jump functions'', namely functions $u\in BV(\Omega)$ whose distributional derivative is purely atomic. For any such function we set
$$\J_{1/2}(\Omega,u):=\sum_{x \in S_u \cap \Omega} |u^+(x)-u^-(x)|^{1/2},$$
where $S_u$ denotes the set of jump points of $u$, while $u^+(x)$ and $u^-(x)$ denote the limits of $u(y)$ as $y\to x^+$ and $y\to x^-$, respectively, so that $|u^+(x)-u^-(x)|$ is the height of the jump located in $x$.

We extend all these functionals to $L^{1}(\Omega)$ by setting them equal to $+\infty$ outside their ``natural domain''. With this understanding it is well-known (see \cite{2008-TAMS-BF,FastPM-CdV} for precise statements) that 
\begin{equation}
    \glim_{\ep\to 0} \RPM_\ep (\Omega,u)=
    \frac{16}{\sqrt{3}}\cdot \J_{1/2}(\Omega,u).
    \label{th:Gamma-conv}
\end{equation}

\paragraph{\textmd{\textit{Strict convergence}}}

Let $(a,b)\subseteq\re$ be an interval. A sequence $\{u_n\}\subseteq BV((a,b))$ is said to converge strictly to some $u_\infty\in BV((a,b))$ if $u_n\to u_\infty$ in $L^1((a,b))$, and the total variation of $u_n$ in $(a,b)$ converges to the total variation of $u_\infty$ in $(a,b)$. In this case, following~\cite{FastPM-CdV} we write $u_n\auto u_\infty$ in $BV((a,b))$.

We refer to~\cite{AFP} for more details concerning strict convergence. Here we just recall that strict convergence implies convergence in $L^p((a,b))$ for every $p\in[1,+\infty)$ (but not necessarily for $p=+\infty$), and also uniform convergence in every closed interval $[c,d]\subseteq(a,b)$ that does not contain jump points of the limit function $u_\infty$.

Finally, we say that $u_n\auto u_\infty$ in $BV\loc(\re)$ if $u_n\auto u_\infty$ in $BV((a,b))$ for every interval $(a,b)\subseteq\re$ whose endpoints are not jump points of the limit function $u_\infty$. In this case it is intended that the limit $u_\infty$ is defined on the whole real line, and for every interval $(a,b)$ the functions $u_n$ are defined at least in $(a,b)$ when $n$ is large enough (how large might depend on the interval).

\paragraph{\textmd{\textit{Staircases}}}

Let us recall the notation introduced in~\cite{FastPM-CdV} in order to describe staircase-like functions (\cite[Definitions~2.3 and~2.4]{FastPM-CdV}).

\begin{defn}[Staircases]
\begin{em}

Let $S:\re\to\re$ be the function defined by
\begin{equation*}
S(x):=2\left\lfloor\frac{x+1}{2}\right\rfloor
\qquad
\forall x\in\re,
\end{equation*}
where, for every real number $\alpha$, the symbol $\lfloor\alpha\rfloor$ denotes the greatest integer less than or equal to $\alpha$. 
\begin{itemize}

\item For every pair $(H,V)$ of real numbers, with $H>0$, we call \emph{canonical $(H,V)$-staircase} the function $S_{H,V}:\re\to\re$ defined by
\begin{equation}
S_{H,V}(x):=V\cdot S(x/H)
\qquad
\forall x\in\re.
\nonumber
\end{equation}

\item We say that $v$ is an \emph{oblique translation} of $S_{H,V}$, and we write $v\in\obl(H,V)$, if there exists a real number $\tau_{0}\in[-1,1]$ such that
\begin{equation}
v(x)=S_{H,V}(x-H\tau_{0})+V\tau_{0}
\qquad
\forall x\in\re.
\nonumber
\end{equation}

\item We say that $v$ is a \emph{graph translation of horizontal type} of $S_{H,V}$, and we write $v\in\orz(H,V)$, if there exists a real number $\tau_{0}\in[-1,1]$ such that
\begin{equation}
v(x)=S_{H,V}(x-H\tau_{0})
\qquad
\forall x\in\re.
\nonumber
\end{equation}

\item We say that $v$ is a \emph{graph translation of vertical type} of $S_{H,V}$, and we write $v\in\vrt(H,V)$, if there exists a real number $\tau_{0}\in[-1,1]$ such that
\begin{equation}
v(x)=S_{H,V}(x-H)+V(1-\tau_{0})
\qquad
\forall x\in\re.
\nonumber
\end{equation}

\end{itemize}

\end{em}
\end{defn}

Roughly speaking, the graph of $S_{H,V}$ is a staircase with steps of horizontal length $2H$ and vertical height $2V$. The origin is the midpoint of the horizontal part of one of the steps. The staircase degenerates to the null function when $V=0$, independently of the value of~$H$. Graph translations correspond to moving the origin in different points of the graph, oblique translations to moving the origin along the line $Hy=Vx$. We refer to~\cite[Remark~2.5 and Figure~2]{FastPM-CdV} for a more detailed description.

\paragraph{\textmd{\textit{First order analysis of minimizers}}}

We are now ready to recall some of the results proved in~\cite{FastPM-CdV} for minimizers of the functional $\PMF_\ep(\beta,f,(a,b),u)$.

First of all, the existence and regularity of minimizers can be established through a standard application of the direct method in the calculus of variation (see \cite[Proposition~2.1]{FastPM-CdV}). To this end, it is enough to assume that $f\in L^2((a,b))$.

On the contrary, the asymptotic behavior of the minimum values seems to depend on the regularity of $f$ (see \cite[Open problem~3]{FastPM-CdV}). In particular, when $f\in C^1([a,b])$ one can prove that (\ref{th:min-asympt}) is true (see \cite[Theorem~2.2]{FastPM-CdV}).

When $\ep\to 0^+$, minimizers develop a staircase-like microstructure, as described in~\cite[Theorem~2.9]{FastPM-CdV}. In a few words, what happens is that the families of blow-ups (\ref{defn:vep}) and (\ref{defn:wep-old}) are relatively compact, and any limit point is a translation of some canonical staircase, with parameters depending on $f'(x_0)$. The rigorous statement is the following.

\begin{thmbibl}[Behavior of minimizers at scale $\omep$]\label{thm:BU}

Let us consider the family of functionals $\PMF_{\ep}(\beta,f,(0,1),u)$ defined by (\ref{defn:SPM}), where $\ep\in(0,1)$ and $\beta>0$ are two real numbers, and $f\in C^{1}([0,1])$ is a given function. Let $\{u_{\ep}\}\subseteq H^{2}((0,1))$ be a family of minimizers, and let $x_{\ep}\to x_{0}\in(0,1)$ be a family of points.

Let $\omep$ be defined by (\ref{defn:omep}), and let us consider the canonical $(H,V)$-staircase with parameters
\begin{equation}
H:=\left(\frac{24}{\beta^{2}|f'(x_{0})|^{3}}\right)^{1/5},
\qquad\qquad
V:=f'(x_{0})H,
\label{defn:HV}
\end{equation}
with the agreement that this staircase is identically equal to~0 when $f'(x_{0})=0$. 

Then the following statements hold true.

\begin{enumerate}
\renewcommand{\labelenumi}{(\arabic{enumi})}

\item \emph{(Compactness of blow-ups).} The family $\{v_{\ep}(y)\}$ defined by (\ref{defn:vep}) is relatively compact with respect to locally strict convergence, and every limit point is a graph translation of the canonical $(H,V)$-staircase.

More precisely, for every sequence $\{\ep_{n}\}\subseteq(0,1)$ with $\ep_{n}\to 0^{+}$ there exist an increasing sequence $\{n_{k}\}$ of positive integers and a function $v_{0}\in\orz(H,V)\cup\vrt(H,V)$ such that
\begin{equation}
v_{\ep_{n_{k}}}(y)\auto v_{0}(y)
\quad\text{in }BV\loc(\re).
\nonumber
\end{equation}

\item \emph{(Realization of all possible graph translations)}. Let $v_{0}\in\orz(H,V)\cup\vrt(H,V)$ be any graph translation of the canonical $(H,V)$-staircase. 

Then there exists a family $\{x'_{\ep}\}\subseteq(0,1)$ such that
\begin{equation*}
\limsup_{\ep\to 0^{+}}\frac{|x'_{\ep}-x_{\ep}|}{\omep}\leq H,
\end{equation*}
and
\begin{equation}
\frac{u_{\ep}(x'_{\ep}+\omep y)-u_{\ep}(x'_{\ep})}{\omep}\auto v_{0}(y)
\quad\text{in }BV\loc(\re).
\nonumber
\end{equation} 

\end{enumerate}

\end{thmbibl}

In the proof of~\cite[Theorem~2.9]{FastPM-CdV} we additionally derive the following two facts that are crucial for the higher order analysis of this paper. 

\begin{thmbibl}[Further properties of blow-ups]\label{thmbibl:add-on}

Under the same assumptions of Theorem~\ref{thm:BU}, let us assume that $v_{\ep_n}(y)\auto v_0(y)$ in $BV\loc(\re)$ for some sequence $\ep_n\to 0^+$.

Then the following additional facts hold true.

    \begin{enumerate}
    \renewcommand{\labelenumi}{(\arabic{enumi})}
    
       \item \emph{(Blow-ups are recovery sequences).} For every bounded interval $(a,b)\subseteq\re$ whose endpoints are not jump points of $v_0$ it holds that
        \begin{equation}\label{recovery_RPM}
        \lim_{n\to+\infty} \RPM_{\ep_{n}}((a,b),v_{\ep_{n}}) = 
        \frac{16}{\sqrt{3}}\cdot \J_{1/2}((a,b),v_0),
        \end{equation}
        namely $v_{\ep_n}$ is a recovery sequence of $v_0$ for the Gamma-converge result (\ref{th:Gamma-conv}) in the interval $(a,b)$.

        \item \emph{(Asymptotic behavior of the fidelity term).} It turns out that
        \begin{equation}
        \frac{f(x_{\ep_n}+\omepn y)-u_{\ep_n}(x_{\ep_n})}{\omepn}\to f'(x_0)y+a_0
        \nonumber
        \end{equation}
        uniformly on bounded sets, where the constant $a_0$ is such that the limiting straight-line passes through all the midpoints of the staircase steps of $v_0$. 

    \end{enumerate}
\end{thmbibl}

The formal proofs of Theorems~\ref{thm:BU} and~\ref{thmbibl:add-on} are extensive. However, for the reader's convenience, it may be helpful to provide a brief overview of the underlying idea, which is relatively straightforward and offers insight into the process.

We start by considering the family of blow-ups $\{w_\ep\}$ defined by (\ref{defn:wep-old}). Since $u_\ep$ is a minimizer of $\PMF_\ep$, with a change of variable one discovers that $w_\ep$ is a local minimizer (namely a minimizer up to perturbations with compact support) of the rescaled version $\RPMF_\ep$ in every admissible interval, with a new forcing term 
\begin{equation}
    g_{\ep}(y):= \frac{f(x_{\ep}+\omep y)-f(x_{\ep})}{\omep},
\nonumber
\end{equation}
which is the corresponding blow-up of $f$.

When $\ep\to 0^+$ it turns out that $g_\ep(y)\to f'(x_0)y$ uniformly on bounded sets, and $\RPMF_\ep$ Gamma-converges to
\begin{equation}
    \frac{16}{\sqrt{3}}\J_{1/2}(\Omega,w)+
    \beta\int_\Omega\left(w(y)-f'(x_0)y\right)^2\,dy,
\nonumber
\end{equation}
and hence it is reasonable that $w_\ep$ converges to some local minimizer of the latter. These local minimizers can be explicitly computed, and are exactly the oblique translations of the canonical staircase with parameters given by (\ref{defn:HV}).

The results for the family $\{v_\ep\}$ follows from the results for the family $\{w_\ep\}$ by simply observing that $v_\ep(y)=w_\ep(y)-w_\ep(0)$.


\setcounter{equation}{0}
\section{Statement of main results -- Higher order analysis}
\label{sec:main}

In this paper we give a closer look at the minimizers of (\ref{defn:SPM}), in order to reveal a finer structure of horizontal and vertical parts of the steps. To accomplish this, we further expand $u_\ep$, or preferably $v_\ep$, whose graph is equivalent up to rescaling. Thanks to statement~(2) of Theorem~\ref{thm:BU} we can assume, without loss of generality, that the whole family $\{v_\ep\}$ converges to a limit staircase $v_0$ for which the origin is the midpoint of either the vertical or the horizontal segment of some step.

We start with the vertical regime, for which we introduce the following notion.
\begin{defn}[Cubic connection]\label{defn:cubic}
\begin{em}

For every pair $(\Lambda,V)$ of positive real numbers, we call \emph{$(\Lambda,V)$-cubic connection} the function $C_{\Lambda,V}:\re\to\re$ defined by
\begin{equation}
C_{\Lambda,V}(x):=
\begin{cases}
-V & \text{if }x\leq -\Lambda, 
\\[1ex]
\displaystyle{\frac{V}{2\Lambda^3}(3\Lambda^2 x-x^3)} & \text{if }x\in[-\Lambda,\Lambda],  
\\[1.5ex]
V  &  \text{if }x\geq\Lambda.
\end{cases}
\nonumber
\end{equation}

\end{em}
\end{defn}

In words, the $(\Lambda,V)$-cubic connection is the unique polynomial of degree three that interpolates the constants $-V$ and $+V$ in a $C^{1}$ way in the interval $[-\Lambda,\Lambda]$. 

The following is the first main result of this paper. It concerns the vertical components of the steps, and hence here we assume that $v_\ep$ converges to the vertical graph translation of the canonical $(H,V)$-staircase with the origin in the middle of a vertical segment.

\begin{thm}[Vertical parts]\label{thm:vertical}

Let us consider the same setting of Theorem~\ref{thm:BU}.

Let us assume in addition that $f'(x_0)\neq 0$, and that 
\begin{equation}
    v_\ep(y)\auto S_{H,V}(y-H)+V
    \quad\text{in }BV\loc(\re).
\nonumber
\end{equation}

Then it turns out that
\begin{equation*}
v_\ep(\ep^2 y)\to C_{\Lambda,V}(y)
\quad\mbox{strongly in }H^2 _{loc}(\R),
\end{equation*}
where $C_{\Lambda,V}(y)$ is the cubic connection introduced in Definition~\ref{defn:cubic}, with \begin{equation}
\Lambda:=\frac{\sqrt{3}}{2}\cdot\sqrt{2V}.
\label{defn:Lambda}
\end{equation}
    
\end{thm}

The second main result of this paper concerns horizontal components. Here we assume that $v_\ep$ converges to the canonical $(H,V)$-staircase. 

\begin{thm}[Horizontal parts]\label{thm:horizontal}

Let us consider the same setting of Theorem~\ref{thm:BU}.

Let us assume in addition that $f'(x_0)\neq 0$, and that 
\begin{equation}
    v_\ep(y)\auto S_{H,V}(y)
    \quad\text{in }BV\loc(\re).
    \label{hp:horiz}
\end{equation}

Then the family $\{v_\ep(y)/\omep^2\}$ is relatively compact in $C^1([a,b])$ for every interval $[a,b]\subseteq(-H,H)$, and every limit point $w_0$ is of the form
\begin{equation}
    w_0(x)=-\frac{\beta}{6}f'(x_0)x^3 + \zeta_0 x
    \qquad
    \forall x\in(-H,H),
\label{th:hor-structure}\end{equation}
for some $\zeta_0\in \R$.

\end{thm}

\begin{rmk}[The misterious $\zeta_0$]
\begin{em}

The precise value of the parameter $\zeta_0$ remains uncertain at present. It is highly probable that it depends on higher-order characteristics of the forcing term, such as $f''(x_0)$, and determining its value may necessitate an even more thorough analysis of the junctions between the horizontal and vertical components of the staircase steps. Additionally, it is conceivable that different sequences may converge to distinct limits, corresponding to different values of $\zeta_0$.
    
\end{em}
\end{rmk}

\begin{rmk}[Blow-ups at different resolutions]
\begin{em}

The results of this paper provide a comprehensive understanding of the blow-ups of minimizers within windows of equal length and height across varying scales. Let us consider, for a suitable $x_\ep\to x_0\in(0,1)$ and a suitable $\alpha_\ep\to 0$, the family of functions
\begin{equation}
    \widehat{v}_\ep(y):=
    \frac{u_\ep(x_\ep+\alpha_\ep y)-u_\ep(x_\ep)}{\alpha_\ep}.
\nonumber
\end{equation}

Then the behavior of the family $\{\widehat{v}_\ep\}$ depends on $\alpha_\ep$ as follows (we always assume that $f$ is of class $C^1$).
\begin{itemize}
    \item (Low resolution). If $\alpha_\ep\gg\omep$, then $\widehat{v}_\ep(y)\to f'(x_0)y$ uniformly on bounded subsets of $\re$ (see~\cite[Corollary~2.13]{FastPM-CdV}). 
 
    \item (Standard resolution). If $\alpha_\ep\sim\omep$, then the possible limits of $\widehat{v}_\ep(y)$ are the horizontal and vertical graph translations of the canonical staircase with parameters given by (\ref{defn:HV}).

    \item (High resolution). If $\alpha_\ep\ll\omep$, then the behavior depends on the behavior of the corresponding family at standard resolution. Let us assume that $f'(x_0)\neq 0$.
    \begin{itemize}
        \item If at standard resolution the limit is a staircase for which the origin lies in the interior of a horizontal segment, then $\widehat{v}_\ep(y)\to 0$ uniformly on bounded sets, namely this higher order blow-up is a horizontal line.

        \item If at standard resolution the limit is a staircase for which the origin lies in the interior of a vertical segment, then $\widehat{v}_\ep(y)$ tends in some sense (that could be made precise) to a vertical line.

        \item  If at standard resolution the limit is a staircase for which the origin lies in the corner between a horizontal and a vertical segment, then it is conceivable that intermediate behaviors may also emerge.
    \end{itemize}

    Intermediate behaviors might emerge at high resolution also when $f'(x_0)=0$.
\end{itemize}

\end{em}    
\end{rmk}

We conclude with some comments on the strategy of the proofs.

Theorem~\ref{thm:vertical} above follows from a general property of the recovery sequences for the Gamma-convergence result (\ref{th:Gamma-conv}).  In essence, the vertical result does not rely crucially on the minimality of $u_\ep$, but rather holds true for any family $v_\ep$ that satisfies (\ref{recovery_RPM}), namely that converges to a pure jump function without ``wasting'' energy. 

Here below we state this general property of recovery sequences, because it could be interesting in itself.

\begin{thm}[Structure of recovery sequences at jump points]\label{thm:vertical-step}

Let $L$ be a positive real number, and let $\{v_{\ep}\}\subseteq H^{2}((-L,L))$ be a family of functions such that $v_\ep(0)=0$ for every $\ep\in(0,1)$. 

Let us consider the family of functionals (\ref{defn:RPM}), and let us assume that there exists a real number $V>0$ such that
\begin{equation}
v_{\ep}(y)\auto V\cdot\sgn(y)
\qquad
\text{strictly in }BV((-L,L)),
\label{hp:vn2jump}
\end{equation}
where $\sgn(y)$ denotes the usual sign function, and
\begin{equation}
\lim_{\ep\to 0^+}\RPM_{\ep}((-L,L),v_{\ep})=
\frac{16}{\sqrt{3}}\cdot\sqrt{2V}.
\label{hp:ABGvn}
\end{equation}

Then for every interval $(a,b)\subseteq\re$ it turns out that
\begin{equation}
    v_\ep(\ep^2 y)\to C_{\Lambda,V}(y)
    \qquad
    \text{strongly in }H^2((a,b)).
\nonumber
\end{equation}

\end{thm}

On the contrary, Theorem~\ref{thm:horizontal} heavily relies on the minimality of $u_\ep$ and is thus more intricate in nature. The minimality, and not just the strict convergence or being a recovery sequence, is essential a first time in order to bound each horizontal part of minimizers within a rectangular box with a height proportional to $\omep^3$. Furthermore, in the second step of the proof the minimality is exploited once again to demonstrate that the derivative of minimizers avoids regions where $\log(1+p^2)$ is non-convex. This reduction allows for the application of more classical tools to establish convergence in the context of a singular perturbation of a convex functional (see Proposition~\ref{lemma:ECV2}).


\setcounter{equation}{0}
\section{Technical preliminaries}\label{sec:lemmata}

The first result that we prove is a sort of quantitative version of the subadditivity of the square root. The intuitive idea is that, when the square root of a sum closely approximates the sum of square roots, then the sum contains a single prominent term.

\begin{lemma}[Quantitative subadditivity of the square root]\label{lemma:MRS}

Let $I$ be a finite or countable nonempty set of indices, and let $f:I\to[0,+\infty)$ be a function. Let us assume that the following three quantities are finite:
\begin{equation*}
S:=\sum_{i\in I}f(i),
\qquad\quad
R:=\sum_{i\in I}\sqrt{f(i)},
\qquad\quad
M:=\max\{f(i):i\in I\}.
\end{equation*}

Then it turns out that
\begin{equation}
\sqrt{S-M}\leq 3\left(R-\sqrt{S}\right).
\label{th:S-M}
\end{equation}

\end{lemma}

\begin{proof}

We distinguish two cases.

\subparagraph{\textmd{\textit{Case 1: $M\geq S/2$}}}

Let $i_{0}\in I$ be such that $f(i_{0})=M$, and let $I':=I\setminus\{i_{0}\}$. From the subadditivity of the square root we deduce that
\begin{equation*}
R=
\sqrt{f(i_{0})}+\sum_{i\in I'}\sqrt{f(i)}\geq
\sqrt{f(i_{0})}+\left(\sum_{i\in I'}f(i)\right)^{1/2}=
\sqrt{M}+\sqrt{S-M}.
\end{equation*}

Now we observe that the function $x\mapsto 1+\sqrt{x}-\sqrt{1+x}$ is increasing, and therefore
\begin{equation*}
1+\sqrt{x}-\sqrt{1+x}\geq 2-\sqrt{2}\geq\frac{1}{3}
\qquad
\forall x\geq 1.
\end{equation*}

Setting $x:=M/(S-M)$ (which is greater than or equal to~1 in this case), from the combination of these two inequalities we conclude that
\begin{equation*}
R-\sqrt{S}\geq\sqrt{S-M}+\sqrt{M}-\sqrt{S}\geq\frac{1}{3}\sqrt{S-M},
\end{equation*}
which proves (\ref{th:S-M}) in this case.

\subparagraph{\textmd{\textit{Case 2: $M\leq S/2$}}}

In this case there exists $I_{1}\subseteq I$ such that
\begin{equation*}
S_{1}:=\sum_{i\in I_{1}}f(i)\in\left[\frac{S}{4},\frac{3S}{4}\right].
\end{equation*}

Setting $I_{2}:=I\setminus I_{1}$, as before from the subadditivity of the square root we obtain that
\begin{equation*}
R=
\sum_{i\in I_{1}}\sqrt{f(i)}+\sum_{i\in I_{2}}\sqrt{f(i)}\geq
\left(\sum_{i\in I_{1}}f(i)\right)^{1/2}+
\left(\sum_{i\in I_{2}}f(i)\right)^{1/2}\geq
\sqrt{S_{1}}+\sqrt{S-S_{1}}.
\end{equation*}

Now we observe that $x\mapsto\sqrt{x}+\sqrt{1-x}-1$ is a concave function with equal values in $1/4$ and $3/4$, and in particular
\begin{equation*}
\sqrt{x}+\sqrt{1-x}-1\geq
\frac{\sqrt{3}-1}{2}\geq\frac{1}{3}
\qquad
\forall x\in\left[\frac{1}{4},\frac{3}{4}\right].
\end{equation*}

Setting $x:=S_{1}/S$ (which lies in the interval $[1/4,3/4]$), from the combination of these two inequalities we conclude that
\begin{equation*}
R-\sqrt{S}\geq\sqrt{S_{1}}+\sqrt{S-S_{1}}-\sqrt{S}\geq\frac{1}{3}\sqrt{S}\geq\frac{1}{3}\sqrt{S-M},
\end{equation*}
which proves (\ref{th:S-M}) in this case.
\end{proof}



\subsection{Estimates for the rescaled Perona-Malik functional}

In this subsection we collect some technical estimates for the singular perturbation of the Perona-Malik functional.

The first result concerns the minimization problem for the term with the second order derivatives. The proof is elementary (see~\cite[Lemma~6.2]{FastPM-CdV}).

\begin{lemma}\label{lemma:ABG}

Let $(a,b)\subseteq\re$ be an interval, and let ${A}_{0}$, ${A}_{1}$, ${B}_{0}$, ${B}_{1}$ be four real numbers. Let us consider the minimum problem
\begin{equation*}
\min\left\{\int_{a}^{b}u''(x)^{2}\,dx:
u\in H^{2}((a,b)),\ \left(u(a),u'(a),u(b),u'(b)\strut\right)=({A}_{0},{A}_{1},{B}_{0},{B}_{1})\right\}.
\end{equation*}

Then the unique minimum point is the unique cubic polynomial that satisfies the four boundary conditions, and the minimum value is
\begin{equation*}
\frac{({B}_{1}-{A}_{1})^{2}}{b-a}+
\frac{12}{(b-a)^{3}}\left[({B}_{0}-{A}_{0})-\frac{{A}_{1}+{B}_{1}}{2}(b-a)\right]^{2}.
\end{equation*}

\end{lemma}



In the second result we consider a general Perona-Malik functional of the form
\begin{equation}
    \PM_{\lambda,\mu,\nu}(\Omega, u) :=\int_\Omega \left[\lambda u''(x)^2 + \mu \log(1+\nu u'(x)^2) \right]dx,
    \label{defn:Fabc}
\end{equation}
where $\Omega\subseteq\re$ is any measurable set, $u\in H^2(\Omega')$ for some open set $\Omega'\supseteq\Omega$, and $\lambda$, $\mu$, $\nu$ are three positive real parameters.

In the following result we show some relations between this functional and the total variation of $u$ (for similar arguments we refer to~\cite[Lemma~A.1]{FastPM-CdV} or to the first part of the proof of~\cite[Proposition~3.3]{ABG}). 

\begin{lemma}[Perona-Malik functional vs total variation]\label{lemma:stima-ABGabc}
Let $\lambda$, $\mu$, $\nu$ be three positive real numbers, let $\PM_{\lambda,\mu,\nu}(\Omega, u)$ be the functional defined in (\ref{defn:Fabc}), let $(a,b)\subseteq\R$ be an interval, and let $u \in H^2((a,b))$. 

For every pair of real numbers $0<D\leq M$, let us consider the regions where the derivative of $u$ is high/small/intermediate, defined as
\begin{gather*}
    \mathcal{A}_M:=\{x \in (a,b) : |u'(x)|>M\},
    \qquad
    \mathcal{C}_D:=\{x \in (a,b) : |u'(x)|<D\},
    \\[0.5ex]
    \mathcal{B}_{D,M}:=\{x \in (a,b) : D\leq|u'(x)|\leq M\}.
\end{gather*}

Then in each region the total variation of $u$ and the value of the functional satisfy the following inequalities.
\begin{enumerate}
\renewcommand{\labelenumi}{(\arabic{enumi})}
    \item In the region with small derivatives it turns out that
    \begin{equation}
        \int_{\mathcal{C}_D} \mu \log\left(1+\nu u'(x)^2\right)dx\geq
        \frac{\mu\log\left(1+\nu D^2\right)}{D^2(b-a)}
        \left(\int_{\mathcal{C}_D}|u'(x)|\,dx\right)^2.
        \label{th:small-der}
    \end{equation}

    \item  In the region with intermediate derivatives it turns out that
    \begin{equation}
        \int_{\mathcal{B}_{D,M}} \mu \log\left(1+\nu u'(x)^2\right)dx\geq
        \frac{\mu\log\left(1+\nu D^2\right)}{M}\int_{\mathcal{B}_{D,M}}|u'(x)|\,dx.
        \label{th:interm-der}
    \end{equation}

    \item  In the region with high derivatives it turns out that
    \begin{equation}
        \int_{\mathcal{A}_M} \mu \log\left(1+\nu u'(x)^2\right)dx\geq
        \mu\log\left(1+\nu M^2\right)\cdot|\mathcal{A}_M|,
        \label{th:est-AM}
    \end{equation}
    where $|\mathcal{A}_M|$ denotes the Lebesgue measure of $\mathcal{A}_M$.

    If in addition $|u'(a)|<M$ and $|u'(b)|<M$, then in every connected component $(c,d)$ of $\mathcal{A}_M$ it turns out that
   \begin{equation}
        \PM_{\lambda,\mu,\nu}((c,d), u)\geq
        \Gamma\left\{\int_{c}^{d}|u'(x)|\,dx - M(d-c)\right\}^{1/2},
        \label{th:high-deriv-cc}
    \end{equation}
    where
    \begin{equation}
        \Gamma:=4\sqrt{\frac{2}{3}} 
        \left[\lambda\mu^{3}\log^3(1+\nu M^2)\right]^{1/4},
    \nonumber
    \end{equation}
    and as a consequence
    \begin{equation}
        \PM_{\lambda,\mu,\nu}(\mathcal{A}_{M}, u)\geq\Gamma
        \left\{\int_{\mathcal{A}_{M}}|u'(x)|\,dx - M|\mathcal{A}_M|\right\}^{1/2}.
        \label{th:high-deriv}
    \end{equation}

\end{enumerate}

\end{lemma}

\begin{proof}

From the concavity of the logarithm we deduce that
\begin{equation}
    \log(1+\nu u'(x)^2)\geq
    \frac{\log(1+\nu D^2)}{\nu D^2}\cdot\nu u'(x)^2
    \qquad
    \forall x\in \mathcal{C}_D,
\nonumber
\end{equation}
and therefore
\begin{eqnarray*}
    \int_{\mathcal{C}_D} \mu \log\left(1+\nu u'(x)^2\right)dx
    & \geq &
    \frac{\mu\log(1+\nu D^2)}{D^2}\int_{\mathcal{C}_D}u'(x)^2\,dx
    \\
    & \geq &
    \frac{\mu\log\left(1+\nu D^2\right)}{D^2}\cdot\frac{1}{|\mathcal{C}_D|}\cdot\left(\int_{\mathcal{C}_D}|u'(x)|\,dx\right)^2,
    \qquad
\end{eqnarray*}
which implies (\ref{th:small-der}) because $|\mathcal{C}_D|\leq b-a$.

In the region with intermediate derivatives we observe that
\begin{equation}
    \int_{\mathcal{B}_{D,M}} \mu \log\left(1+\nu u'(x)^2\right)dx\geq
    \mu\log\left(1+\nu D^2\right)\cdot|\mathcal{B}_{D,M}|
    \label{est:BM}
\end{equation}
and
\begin{equation}
    \int_{\mathcal{B}_{D,M}}|u'(x)|\,dx\leq
    |\mathcal{B}_{D,M}|\cdot M.
\nonumber
\end{equation}

By combining these two inequalities we obtain (\ref{th:interm-der}).

In the region with high derivatives, the proof of (\ref{th:est-AM}) is analogous to the proof of (\ref{est:BM}) above.

Let us assume now that $|u'(a)|$ and $|u'(b)|$ are both less than $M$, and let $(c,d)$ be a connected component of $\mathcal{A}_M$. Then $c> a$ and $d<b$, and from the continuity of $u'$ we deduce that $u'$ has constant sign in $(c,d)$ and $|u'(c)|=|u'(d)|=M$, so that
$$\int_{c}^{d} |u'(x)|\,dx = |u(d)-u(c)|. $$

As a consequence, from Lemma~\ref{lemma:ABG} we deduce that
$$\int_{c}^{d} u''(x)^2\,dx\geq \frac{12}{(d-c )^3}\left(|u(d)-u(c)|-M(d-c )\strut\right)^2,$$
and hence
\begin{eqnarray*}
\PM_{\lambda,\mu,\nu}((c,d), u) & \geq &
\frac{12\lambda}{(d-c )^3}
\left(|u(d)-u(c)|-M(d-c )\strut\right)^2 
+\mu\log(1+\nu M^2)(d-c )
\\
& \geq & 
\frac{4}{3^{3/4}}\left[12\lambda\mu^{3}\log^3(1+\nu M^2)\right]^{1/4} 
\left\{\int_{c}^{d} |u'(x)|\,dx-M(d-c )\right\}^{1/2},
\end{eqnarray*}
where in the last step we exploited the inequality
$$P+Q\geq \frac{4}{3^{3/4}}(PQ^3)^{1/4}\qquad \forall (P,Q)\in [0,+\infty)^2.$$

This proves (\ref{th:high-deriv-cc}). Summing over all connected components of $\mathcal{A}_M$, and exploiting the subadditivity of the square root, we obtain (\ref{th:high-deriv}).
\end{proof}


The following result provides an estimate from below for the rescaled Perona-Malik functional in terms of the total variation of the argument.  It suggests that achieving a certain total variation can be best done either by using a straight line, in which case the cost is determined by the linearization of the term with the logarithm, or by utilizing a suitable regularization of a function with a single jump, in which case the cost is determined by the Gamma-limit, specifically the square root of the jump height. 

\begin{lemma}[Bound from below for the rescaled Perona-Malik functional]\label{lemma:RPM>=}

Let us consider an interval $(a,b)\subseteq\re$, a real number $\ep\in(0,1)$, and a function $v\in H^2((a,b))$ such that $|v'(a)|<1/\ep$ and $|v'(b)|<1/\ep$. Let
\begin{equation}
    TV(v):=\int_a^b|v'(x)|\,dx
\nonumber
\end{equation}
denote the total variation of $v$ in $(a,b)$, and let $\RPM_\ep$ denote the rescaled Perona-Malik functional defined in (\ref{defn:RPM}).

Then it turns out that
\begin{equation}
    \RPM_\ep((a,b),v)\geq
    \min\left\{\frac{2^{13/4}}{3}\cdot[TV(v)]^{1/2},\frac{\log 2}{9(b-a)}\cdot\frac{[TV(v)]^2}{\omep^2}\right\}
    \label{th:RPM-below}
\end{equation}
for every $\ep$ small enough, and specifically for every $\ep\in(0,1)$ such that
\begin{equation}
    \frac{\log 2}{3}\cdot\frac{\ep}{\omep^2}\geq
    \left[\frac{64\sqrt{2}}{9}\cdot\frac{\log 2}{9(b-a)}\right]^{1/3}\cdot\frac{1}{\omep^{2/3}}.
    \label{hp:ep-small}
\end{equation}
 
\end{lemma}

\begin{proof}

We write $(a,b)$ as the union of the three sets
\begin{gather*}
    \mathcal{A}_\ep:=\left\{x\in(a,b):|v'(x)|>\frac{1}{\ep}\right\},
    \qquad
    \mathcal{B}_\ep:=\left\{x\in(a,b):1\leq|v'(x)|\leq\frac{1}{\ep}\right\},
    \\
    \mathcal{C}_\ep:=\left\{x\in(a,b):|v'(x)|<1\right\},
\end{gather*}
and we observe that we are in the same framework of Lemma~\ref{lemma:stima-ABGabc} with
\begin{equation}
    \lambda:=\ep^6,
    \qquad
    \mu:=\frac{1}{\omep^2},
    \qquad
    \nu:=1,
    \qquad
    D:=1,
    \qquad
    M:=\frac{1}{\ep}.
    \label{choices:RPM-est}
\end{equation}

Now we write the total variation of $v$ as
\begin{equation}
    \int_a^b|v'(x)|\,dx=
    \int_{\mathcal{C}_\ep}|v'(x)|\,dx+
    \left(\int_{\mathcal{A}_\ep}|v'(x)|\,dx-\frac{|\mathcal{A}_\ep|}{\ep}\right)+
    \left(\int_{\mathcal{B}_\ep}|v'(x)|\,dx+\frac{|\mathcal{A}_\ep|}{\ep}\right),
\nonumber
\end{equation}
and we deduce that at least one of the three terms in the right-hand side is greater than or equal to 1/3 of the left-hand side. Let us distinguish the three cases.

\begin{itemize}
    \item In the first case from (\ref{th:small-der}) with the choices (\ref{choices:RPM-est}) we deduce that
\begin{equation}
    \quad
    \RPM_\ep(\mathcal{C}_\ep,v)\geq
    \frac{1}{\omep^2}\cdot\frac{\log 2}{b-a}\left(\int_{\mathcal{C}_\ep}|v'(x)|\,dx\right)^2
    \geq
    \frac{\log 2}{b-a}\cdot\frac{1}{9}\cdot\frac{[TV(v)]^2}{\omep^2},
    \quad
\nonumber
\end{equation}
which implies (\ref{th:RPM-below}) in this first case.

\item In the second case from (\ref{th:high-deriv}) with the choices (\ref{choices:RPM-est}) we deduce that
\begin{eqnarray*}
    \RPM_\ep(\mathcal{A}_\ep,v) & \geq &
    4\sqrt{\frac{2}{3}}\left[\frac{\log(1+1/\ep^2)}{|\log\ep|}\right]^{3/4}
    \left\{\int_{\mathcal{A}_\ep}|v'(x)|\,dx-\frac{|\mathcal{A}_\ep|}{\ep}\right\}^{1/2}
    \\[1ex]
    & \geq &
    4\sqrt{\frac{2}{3}}\cdot 2^{3/4}\cdot\frac{1}{\sqrt{3}}\cdot[TV(v)]^{1/2},
\end{eqnarray*}
which implies (\ref{th:RPM-below}) in this second case.

\item  In the third case we start by observing that
\begin{equation}
    \min\left\{\frac{2^{13/4}}{3}\cdot\frac{1}{\sqrt{x}},\frac{\log 2}{9(b-a)}\cdot\frac{x}{\omep^2}\right\}\leq
    \left[\frac{64\sqrt{2}}{9}\cdot\frac{\log 2}{9(b-a)}\right]^{1/3}\cdot\frac{1}{\omep^{2/3}}
\nonumber
\end{equation}
for every $x>0$, because the left-hand side is maximum when the two terms in the minimum coincide. Applying this inequality with $x:=TV(v)$, from (\ref{hp:ep-small}) we deduce that
\begin{equation}
    \frac{\log 2}{3}\cdot\frac{\ep}{\omep^2}\geq
    \min\left\{\frac{2^{13/4}}{3}\cdot\frac{1}{[TV(v)]^{1/2}},\frac{\log 2}{9(b-a)}\cdot\frac{TV(v)}{\omep^2}\right\}.
\nonumber
\end{equation}

At this point from (\ref{th:interm-der}) and (\ref{th:est-AM}) with the choices (\ref{choices:RPM-est}) we conclude that
\begin{multline*}
    \RPM_\ep(\mathcal{B}_\ep,v)+\RPM_\ep(\mathcal{A}_\ep,v)\geq
    \log 2\cdot\frac{\ep}{\omep^2}\left(\int_{\mathcal{B}_\ep}|v'(x)|\,dx+\frac{|\mathcal{A}_\ep|}{\ep}\right)
    \\
    \geq\min\left\{\frac{2^{13/4}}{3}\cdot\frac{1}{[TV(v)]^{1/2}},\frac{\log 2}{9(b-a)}\cdot\frac{TV(v)}{\omep^2}\right\}\cdot TV(v),
\end{multline*}
which implies (\ref{th:RPM-below}) also in the last case.
\qed

\end{itemize}
\renewcommand{\qedsymbol}{}

\end{proof}


In the final result of this subsection, we calculate the values of the rescaled Perona-Malik functional with the fidelity term for both a minimizer and another function that shares the same values at the boundary, but differs in derivative. This yields a sort of one-sided local Lipschitz continuity of the values of the functional around a minimizer, with a constant that diminishes as $\ep\to 0^+$. This conclusion aligns with the intuition that the cost of a minor adjustment, which changes the derivative's values at the boundary, decreases as the singular perturbation term approaches zero. The rate $\ep^{4/3}$ could be slightly improved, but when we apply this result we need only that it is at least $\ep$.

\begin{lemma}[Dependence on boundary data for derivatives]\label{lemma:Lip-BCs}

Let us consider the functional $\RPMF_\ep(\beta,g_\ep,(a,b),v)$ defined by (\ref{defn:RPMF}), where $(a,b)\subseteq\re$ is an interval, $\ep\in(0,1)$ and $\beta\geq 0$ are real numbers, and $g_\ep\in L^2((a,b))$. 

Let $v$ and $\hatv$ be two functions in $H^2((a,b))$, and let us assume that
\begin{enumerate}
\renewcommand{\labelenumi}{(\roman{enumi})}
    \item $\hatv$ is a local minimizer, namely
    \begin{equation}
    \RPMF_\ep(\beta,g_\ep,(a,b),\hatv)\leq \RPMF_\ep(\beta,g_\ep,(a,b),\hatv+\phi)
        \qquad
        \forall\phi\in H^2_0((a,b)),
\nonumber
    \end{equation}

    \item the boundary conditions of $v$ and $\hatv$ satisfy $\hatv(a)=v(a)$, $\hatv(b)=v(b)$, and
    \begin{equation}
        |\hatv'(a)-v'(a)|+|\hatv'(b)-v'(b)|\leq 1,
        \label{hp:Delta'-small}
    \end{equation}

    \item there exists two real numbers $L_0$ and $K_0$ such that
\begin{equation}
    b-a\geq L_0,
    \qquad\text{and}\qquad
    \RPMF_\ep(\beta,g_\ep,(a,b),v)\leq K_0.
    \label{hp:energy-small}
\end{equation}

\end{enumerate}

Then there exists two positive real numbers $\Gamma_0$ and $\ep_0$, that depend only on $L_0$ and $K_0$, such that
\begin{multline}
\qquad
\RPMF_\ep(\beta,g_\ep,(a,b),\hatv)-\RPMF_\ep(\beta,g_\ep,(a,b),v)
\\[0.5ex]
\leq\Gamma_0\ep^{4/3}\left(|\hatv'(a)-v'(a)|+|\hatv'(b)-v'(b)|\strut\right),
\qquad
\label{th:Lip-BCs}
\end{multline}
provided that $\ep\in(0,\ep_0)$.
    
\end{lemma}

\begin{proof}

Let us compare, for any function $r\in H^2((a,b))$, the values of the functional computed in $v$ and $v+r$.

As for the term with the second order derivatives, we obtain that
\begin{eqnarray}
    \lefteqn{\hspace{-2em}
    \ep^6\int_a^b (v''(x)+r''(x))^2\,dx-\ep^6 \int_a^b v''(x)^2\,dx}
\nonumber
    \\
    & = &
    \ep^6 \int_a^b\left(2v''(x)r''(x)+r''(x)^2\right)dx
\nonumber
    \\
    & \leq &
    2\ep^3\left(\ep^6\int_a^b v''(x)^2\,dx\right)^{1/2}\left(\int_a^b r''(x)^2\,dx\right)^{1/2}
    +\ep^6\int_a^b r''(x)^2\,dx.
    \label{est:rep-2}
\end{eqnarray}

As for the term with the first order derivatives, we observe that the function $\sigma\mapsto\log(1+\sigma^2)$ is Lipschitz continuous, with Lipschitz constant equal to~1, and we obtain that
\begin{multline}
    \qquad
    \frac{1}{\omega(\ep)^2}\int_a^b\log\left(1+(v'(x)+r'(x))^2\right)dx-
    \frac{1}{\omega(\ep)^2}\int_a^b\log\left(1+v'(x)^2\right)dx
    \\
    \leq
    \frac{1}{\omega(\ep)^2}\int_a^b|r'(x)|\,dx.
    \qquad
    \label{est:rep-1}
\end{multline}

Finally, for the fidelity term we obtain that
\begin{eqnarray}
    \lefteqn{\hspace{-2em}
    \int_a^b (v(x)+r(x)-g_\ep(x))^2\,dx-\int_a^b (v(x)-g_\ep(x))^2\,dx}
\nonumber
    \\
    & = & 2\int_a^b r(x)(v(x)-g_\ep(x))\,dx+\int_a^b r(x)^2\,dx
\nonumber
    \\
    & \leq &
    2\left(\int_a^b r(x)^2\,dx\right)^{1/2}\left(\int_a^b(v(x)-g_\ep(x))^2\,dx\right)^{1/2}+\int_a^b r^2(x)\,dx.
    \label{est:rep-0}
\end{eqnarray}

Now we consider the function $\psi:[0,+\infty)\to[0,+\infty)$ defined by
\begin{equation}
    \psi(x):=
    \begin{cases}
        x(1-x)^2 & \text{if }x\in[0,1],
        \\
        0 & \text{if }x\geq 1,
    \end{cases}
    \label{defn:psi}
\end{equation}
and we observe that $\psi\in H^2((0,+\infty))$ because the junction in $x=1$ is of class $C^1$. Then we set
\begin{equation}
    \Delta_a:=\hatv'(a)-v'(a),
    \qquad\qquad
    \Delta_b:=\hatv'(b)-v'(b),
\nonumber
\end{equation}
and for $\ep^{10/3}<(b-a)/2$ we introduce the function
\begin{equation}
    r(x):=\ep^{10/3}\Delta_a\cdot\psi\left(\frac{x-a}{\ep^{10/3}}\right)
    -\ep^{10/3}\Delta_b\cdot\psi\left(\frac{b-x}{\ep^{10/3}}\right).
\nonumber
\end{equation}

The function $v+r$ has the same four boundary conditions of $\hatv$, and hence
\begin{equation}
    \RPMF_\ep(\beta,g_\ep,(a,b),\hatv)\leq
    \RPMF_\ep(\beta,g_\ep,(a,b),v+r).
\nonumber
\end{equation}

Therefore, (\ref{th:Lip-BCs}) is proved for every small enough $\ep$ if we show that
\begin{equation}
    \RPMF_\ep(\beta,g_\ep,(a,b),v+r)-\RPMF_\ep(\beta,g_\ep,(a,b),v)\leq
    \Gamma_0 \ep^{4/3}(\Delta_a+\Delta_b).
\nonumber
\end{equation}

To this end, with standard computations we obtain that
\begin{gather*}
    \int_a^b r(x)^2\,dx=\Theta_0\ep^{10}\left(|\Delta_a|^2+|\Delta_b|^2\right),
    \qquad
    \int_a^b |r'(x)|\,dx=\Theta_1\ep^{10/3}\left(|\Delta_a|+|\Delta_b|\strut\right),
    \\
    \int_a^b r''(x)^2\,dx=\frac{\Theta_2}{\ep^{10/3}}\left(|\Delta_a|^2+|\Delta_b|^2\right)
\end{gather*}
for three real constants $\Theta_0$, $\Theta_1$, $\Theta_2$ that depend only on suitable integrals of $\psi$ and its first and second derivative. Plugging all these equalities into (\ref{est:rep-2}), (\ref{est:rep-1}), and (\ref{est:rep-0}), and recalling the conditions (\ref{hp:energy-small}) and (\ref{hp:Delta'-small}), we obtain that (\ref{th:Lip-BCs}) holds true with a suitable constant $\Gamma_0$ when $\ep$ is small enough.
\end{proof}



\subsection{Singular perturbation of convex functionals}

In this subsection, we discuss the asymptotic behavior of the minimizers of a singular perturbation of a convex functional. While the outcome and concepts seem to be rooted in classical theory (as evidenced, for instance, by~\cite[Chapter~II]{1973-JLL}), we encountered difficulty in locating an appropriate citation in existing literature. Hence, for the reader's ease, we present and demonstrate the version essential to our discussion.

The setting is the following. Let us consider six real numbers
\begin{equation}
    a''<a'<a<b<b'<b'',
\nonumber
\end{equation}
and a real number $\ep_0>0$. For every $\ep\in(0,\ep_0)$, let us consider two real numbers $a_\ep\in(a'',a')$ and $b_\ep\in(b',b'')$, a continuous function $h_\ep:[a'',b'']\to\re$, and a function $\varphi_\ep\in C^2(\re)$.

Let us assume that
\begin{enumerate}
\renewcommand{\labelenumi}{(\roman{enumi})}

    \item for every $\sigma\in\re$ it turns out that $\varphi_\ep(\sigma)\to\sigma^2$ as $\ep\to 0^+$,

    \item  the function $\varphi_\ep$ satisfies the uniform convexity assumption
\begin{equation}
    \varphi_\ep''(\sigma)\geq 1
    \qquad
    \forall\sigma\in\re,
    \quad
    \forall \ep\in(0,\ep_0),
    \label{hp:phin-convex}
\end{equation}
and there exists a real number $M_0$ such that
\begin{equation}
    \varphi_\ep''(\sigma)\leq M_0
    \qquad
    \forall\sigma\in\re,
    \quad
    \forall \ep\in(0,\ep_0),
    \label{hp:phin''-above}
\end{equation}

\item there exists a continuous function $h_0:[a'',b'']\to\re$ such that, as $\ep\to 0^+$, 
\begin{equation}
    h_\ep(x)\to h_0(x)
    \qquad
    \text{uniformly in }[a'',b''].
\nonumber
\end{equation}

\end{enumerate} 

Finally, let $\eta_\ep\to 0^+$ and $\tau_\ep\to 0^+$ be two families of positive real numbers.

We consider the family of first order functionals
\begin{equation}
    \GG_\ep((a_\ep,b_\ep),z):=\int_{a_\ep}^{b_\ep}\left[\varphi_\ep(z'(x))+\eta_\ep^2 z(x)^2-2h_\ep(x)z(x)\right]dx,
\nonumber
\end{equation}
and the family of their second order singular perturbations
\begin{equation}
    \FF_\ep((a_\ep,b_\ep),w):=\tau_\ep^2\int_{a_\ep}^{b_\ep}w''(x)^2\,dx+\GG_\ep((a_\ep,b_\ep),w).
    \label{defn:Fn}
\end{equation}

The idea is that minimizers of $\FF_\ep$ tend to minimizers of the functional obtained by setting formally $\ep=0$ in the expression of $\GG_\ep$.

\begin{prop}[Compactness of minimizers for the singular perturbation]\label{lemma:ECV2}

Let us consider the family of functionals $\FF_\ep$ defined by (\ref{defn:Fn}) under the assumptions described above. For every $\ep\in(0,\ep_0)$, let $w_\ep\in H^2((a_\ep,b_\ep))$ be a local minimizer of $\FF_\ep$, namely a function such that
\begin{equation}
    \FF_\ep((a_\ep,b_\ep),w_\ep)\leq \FF_\ep((a_\ep,b_\ep),w_\ep+\phi)
    \qquad
    \forall\phi\in H^2_0((a_\ep,b_\ep)).
\nonumber
\end{equation}

Let us assume that the boundary conditions are uniformly bounded, namely 
\begin{equation}
    \sup\left\{|w_\ep(a_\ep)|+|w_\ep'(a_\ep)|+|w_\ep(b_\ep)|+|w_\ep'(b_\ep)|:\ep\in(0,\ep_0)\strut\right\}<+\infty.
\nonumber
\end{equation}

Then the family $\{w_\ep\}$ is relatively compact in $H^1((a',b'))$ and in $C^1([a,b])$. Moreover, every limit point $w_0$ belongs to $C^2([a',b'])$ and satisfies
\begin{equation}
    w_0''(x)=-h_0(x)
    \qquad
    \forall x\in[a',b'].
\nonumber
\end{equation}

\end{prop}

\begin{proof}

The idea is to write $w_\ep$ as
\begin{equation}
    w_\ep(x)=z_\ep(x)+\theta_\ep(x)+\rho_\ep(x)
    \qquad
    \forall x\in[a_\ep,b_\ep],
\nonumber
\end{equation}
where
\begin{itemize}

\item $z_\ep$ is the (unique, by convexity) minimizer of $\GG_\ep$ with boundary conditions 
\begin{equation}
    z_\ep(a_\ep)=w_\ep(a_\ep),
    \qquad\qquad
    z_\ep(b_\ep)=w_\ep(b_\ep),
\nonumber
\end{equation}

\item  $\theta_\ep$ is the (unique, by convexity) minimizer of the functional
\begin{equation}
    \HH_\ep((a_\ep,b_\ep),\theta):=
    \int_{a_\ep}^{b_\ep}
    H_\ep(x,\theta(x),\theta'(x),\theta''(x))\,dx,
\nonumber
\end{equation}
where
\begin{equation}
    H_\ep(x,s,p,q):=\tau_\ep ^2 q^2+
    \left[\varphi_\ep(z_\ep'(x)+p)-\varphi_\ep(z_\ep'(x))-\varphi_\ep'(z_\ep'(x))\cdot p\right]+
    \eta_\ep^2 s^2,
\nonumber
\end{equation}
with boundary conditions
\begin{equation}
    \theta_\ep(a_\ep)=\theta_\ep(b_\ep)=0,
    \qquad
    \theta_\ep'(a_\ep)=w_\ep'(a_\ep)-z_\ep'(a_\ep),
    \qquad
    \theta_\ep'(b_\ep)=w_\ep'(b_\ep)-z_\ep'(b_\ep),
\nonumber
\end{equation}

\item  $\rho_\ep$ is a suitable remainder, that satisfies necessarily
\begin{equation}
    \rho_\ep(a_\ep)=\rho_\ep(b_\ep)=\rho_\ep'(a_\ep)=\rho_\ep'(b_\ep)=0.
    \label{BC:theta_ep}
\end{equation}

\end{itemize}

Roughly speaking, $z_\ep$ is the solution to the problem without the singular perturbation term, $\theta_\ep$ takes into account the loss of the two boundary conditions on the derivative, and $\rho_\ep$ is what remains.

We claim the following.
\begin{enumerate}
    \item The family $\{z_\ep\}$ is bounded in $C^2([a_\ep,b_\ep])$, and in particular its restriction to $[a',b']$ is relatively compact in $C^1([a',b'])$. Moreover, any limit point $z_0$ belongs to $C^2([a',b'])$ and satisfies
\begin{equation}
    z_0''(x)=-h_0(x)
    \qquad
    \forall x\in[a',b'].
    \label{th:ECV1}
\end{equation}

\item  The family $\{\theta_\ep\}$ is bounded in $C^1([a_\ep,b_\ep])$, and 
\begin{equation}
    \theta_\ep\to 0
    \qquad
    \text{both in }H^1((a',b'))\text{ and in }C^1([a,b]).
\nonumber
\end{equation}

\item  The family $\{\rho_\ep\}$ is bounded in $H^2((a_\ep,b_\ep))$ and
\begin{equation}
    \rho_\ep\to 0
    \qquad
    \text{in }C^1([a',b']). 
\nonumber
\end{equation}

\end{enumerate}

These claims are enough to establish all our conclusions, and therefore in the rest of the proof we can limit ourselves to showing them.

\paragraph{\textmd{\textit{Compactness of $\{z_\ep\}$}}}

This part is rather standard, and so we limit ourselves to sketching the main points. To begin with, we recall that, since we are dealing with convex functions of class $C^2$, the pointwise convergence $\varphi_\ep(\sigma)\to\sigma^2$ implies that the convergence is uniform on compact sets, and also  $\varphi_\ep'(\sigma)\to 2\sigma$ uniformly on compact sets.

Comparing $z_\ep$ with the affine function with the same boundary conditions, we obtain that $\GG_\ep((a_\ep,b_\ep),z_\ep)$ is bounded independently of $\ep$. Thanks to (\ref{hp:phin-convex}), from this bound we deduce that $z_\ep'$ is bounded in $L^2((a_\ep,b_\ep))$, which in turn implies that $z_\ep$ is bounded in $C^0([a_\ep,b_\ep])$ because of the bound on boundary data. The classical regularity theory shows that actually $z_\ep\in C^2([a_\ep,b_\ep])$, and it satisfies the Euler-Lagrange equation \begin{equation}
    \varphi_\ep''(z_\ep'(x))\cdot z_\ep''(x)=2\eta_\ep^2 z_\ep(x)-2h_\ep(x)
    \qquad
    \forall x\in[a_\ep,b_\ep].
\nonumber
\end{equation}

Due to the uniform convexity assumption (\ref{hp:phin-convex}), and to the uniform bounds on $z_\ep$, the latter implies that the family $\{z_\ep\}$ is bounded in $C^2([a_\ep,b_\ep])$, and hence relatively compact in $C^1([a',b'])$.

Finally, if $z_\ep\to z_0$ in $C^1([a',b'])$, even up to subsequences, then we can pass to the limit in the Euler-Lagrange equation, at least when written in its weak form, and we conclude that the 
limit $z_0$ satisfies (\ref{th:ECV1}).

\paragraph{\textmd{\textit{Convergence of $\{\theta_\ep\}$}}}

To begin with, we observe the key property of $\HH_\ep$, namely that
\begin{equation}
    \HH_\ep((a_\ep,b_\ep),\theta)\geq
    \tau_\ep^2\int_{a_\ep}^{b_\ep}\theta''(x)^2\,dx+\frac{1}{2}\int_{a_\ep}^{b_\ep}\theta'(x)^2\,dx
    \qquad
    \forall\theta\in H^2((a_\ep,b_\ep)),
    \label{est:HH-below}
\end{equation}
which follows again from the bound (\ref{hp:phin-convex}). Similarly, from (\ref{hp:phin''-above}) one obtains the bound from above
\begin{equation}
    \HH_\ep((a_\ep,b_\ep),\theta)\leq
    \tau_\ep^2\int_{a_\ep}^{b_\ep}\theta''(x)^2\,dx+
    \frac{M_0}{2}\int_{a_\ep}^{b_\ep}\theta'(x)^2\,dx+
    \eta_\ep^2\int_{a_\ep}^{b_\ep}\theta(x)^2\,dx.
    \label{est:HH-above}
\end{equation}

Now we show that
\begin{equation}
    \HH_\ep((a_\ep,b_\ep),\theta_\ep)=O(\tau_\ep).
\label{est:HH-aep-bep}
\end{equation}

To this end, it is enough to exhibit a competitor $\widehat{\theta}_\ep$, with the same four boundary conditions of $\theta_\ep$, such that $\HH_\ep((a_\ep,b_\ep),\widehat{\theta}_\ep)=O(\tau_\ep)$. For this purpose, we consider again the function $\psi:[0,+\infty)\to[0,+\infty)$ defined by (\ref{defn:psi}), and we set
\begin{equation}
    \widehat{\theta}_{\ep}(x):=
    \tau_\ep(w_\ep'(a_\ep)-z_\ep'(a_\ep))\cdot\psi\left(\frac{x-a_\ep}{\tau_\ep}\right)-
    \tau_\ep(w_\ep'(b_\ep)-z_\ep'(b_\ep))\cdot\psi\left(\frac{b_\ep-x}{\tau_\ep}\right).
\nonumber
\end{equation}

We observe that, when $\tau_\ep<(b-a)/2$, the function $\widehat{\theta}_\ep$ has the same four boundary conditions of $\theta_\ep$. In addition, $\widehat{\theta}_\ep(x)$ is equal to 0 in $(a_\ep,b_\ep)$, with the exception of two intervals of length $\tau_\ep$ at the boundary, where
\begin{equation}
    \widehat{\theta}_\ep(x)=O(\tau_\ep),
    \qquad\quad
    \widehat{\theta}_\ep'(x)=O(1),
    \qquad\quad
    \widehat{\theta}_\ep''(x)=O(1/\tau_\ep).
\nonumber
\end{equation}

When we plug these relations into (\ref{est:HH-above}) we deduce that $\HH_\ep((a_\ep,b_\ep),\widehat{\theta}_\ep)=O(\tau_\ep)$, which proves (\ref{est:HH-aep-bep}). Taking (\ref{est:HH-below}) into account, and recalling that $\theta_\ep$ vanishes at the boundary of $(a_\ep,b_\ep)$, this is enough to conclude that $\theta_\ep\to 0$ in $H^1((a',b'))$. 

In order to improve this convergence, we take again (\ref{est:HH-below}) into account, and we deduce that for every $x\in[a_\ep,b_\ep]$ it turns out that
\begin{multline}
    \qquad
    \tau_\ep\left|\theta_\ep'(x)^2-\theta_\ep'(a_\ep)^2\right|\leq
    \tau_\ep\int_{a_\ep}^{b_\ep}2|\theta_\ep'(t)|\cdot|\theta_\ep''(t)|\,dt
    \\
    \leq
    \int_{a_\ep}^{b_\ep}\left[\tau_\ep^2\theta_\ep''(t)^2+\theta_\ep'(t)^2\right]dt\leq
    2\HH_\ep((a_\ep,b_\ep),\theta_\ep)=
     O(\tau_\ep).    
    \qquad
\label{est:theta-ptwise}
\end{multline}

Since $\theta_\ep'(a_\ep)$ and $\theta_\ep(a_\ep)$ are bounded, this is enough to establish the boundedness of the family $\{\theta_\ep\}$ in $C^1([a_\ep,b_\ep])$.

Now let us consider the interval $(a_\ep,a_\ep+\tau_\ep^{1/2})$. From (\ref{est:HH-below}) and (\ref{est:HH-aep-bep}) we deduce that there exists at least one point $c_\ep$ in this interval such that
\begin{equation}
    |\theta_\ep'(c_\ep)|=O(\tau_\ep^{1/4}),
\nonumber
\end{equation}
and in addition
\begin{equation}
    |\theta_\ep(c_\ep)|\leq
    \int_{a_\ep}^{c_\ep}|\theta_\ep'(x)|\,dx\leq
    |c_\ep-a_\ep|^{1/2}\left\{\int_{a_\ep}^{c_\ep}|\theta_\ep'(x)|^2\,dx\right\}^{1/2}=
     O(\tau_\ep^{3/4}).
\nonumber
\end{equation}

Similarly, there exists $d_\ep\in(b_\ep-\tau_\ep^{1/2},b_\ep)$ such that
\begin{equation}
    |\theta_\ep'(d_\ep)|=O(\tau_\ep^{1/4})
    \qquad\quad\text{and}\quad\qquad
    |\theta_\ep(d_\ep)|=O(\tau_\ep^{3/4}).
\nonumber
\end{equation}

Now we consider the interval $(c_\ep,d_\ep)$, and we claim that
\begin{equation}
    \HH_\ep((c_\ep,d_\ep),\theta_\ep)=O(\tau_\ep^{3/2}),
    \label{est:HH-aep'-bep'}
\end{equation}
which is similar to (\ref{est:HH-aep-bep}), but with a better exponent. To this end, we observe that $\theta_\ep$ is a local minimizer in this interval, and therefore it is enough to exhibit a competitor $\widehat{\theta}_\ep$, with the same four boundary conditions in $(c_\ep,d_\ep)$ for which the estimate holds. This time we can choose 
\begin{equation}
    \widehat{\theta}_\ep(x):=
    \theta_\ep(c_\ep)+P_\ep\cdot(x-c_\ep)+
    Q_\ep\tau_\ep\cdot\psi\left(\frac{x-c_\ep}{\tau_\ep}\right)-    
    D_\ep\tau_\ep\cdot\psi\left(\frac{d_\ep-x}{\tau_\ep}\right),
\nonumber
\end{equation}
where again $\psi$ is defined by (\ref{defn:psi}) and
\begin{equation}
    P_\ep:=\frac{\theta_\ep(d_\ep)-\theta_\ep(c_\ep)}{d_\ep-c_\ep},
    \qquad
    Q_\ep:=\theta_\ep'(c_\ep)-P_\ep,
    \qquad
    D_\ep:=\theta_\ep'(d_\ep)-P_\ep.
\nonumber
\end{equation}

In words, this function is the straight line that connects the boundary conditions, with the usual correction in two intervals of length $\tau_\ep$ at the boundary in order to fit also the boundary conditions on the derivative. One can check that in the two intervals at the boundary it turns out that 
\begin{equation}
    \widehat{\theta}_\ep(x)=O(\tau_\ep^{3/4}),
    \qquad\quad
    \widehat{\theta}_\ep'(x)=O(\tau_\ep^{1/4}),
    \qquad\quad
    \widehat{\theta}_\ep''(x)=O(\tau_\ep^{-3/4}),
\nonumber
\end{equation}
while in the rest of $(c_\ep,d_\ep)$ it turns out that
\begin{equation}
    \widehat{\theta}_\ep(x)=O(\tau_\ep^{3/4}),
    \qquad\quad
    \widehat{\theta}_\ep'(x)=P_\ep=O(\tau_\ep^{3/4}),
    \qquad\quad
    \widehat{\theta}_\ep''(x)=0,
\nonumber
\end{equation}
and these relations imply (\ref{est:HH-aep'-bep'}) because of an estimate of the form (\ref{est:HH-above}) in $(c_\ep,d_\ep)$.

At this point, arguing as in (\ref{est:theta-ptwise}) we find that
\begin{equation}
    \tau_\ep\left|\theta_\ep'(x)^2-\theta_\ep'(c_\ep)^2\right|\leq
     O(\tau_\ep^{3/2})
    \qquad
    \forall x\in [c_\ep,d_\ep].
\nonumber
\end{equation}

Since $\theta_\ep'(c_\ep)\to 0$ and $\theta_\ep(c_\ep)\to 0$, this is enough to conclude that
\begin{equation}
    \lim_{\ep\to 0^+}\max\left\{|\theta_\ep(x)|+|\theta_\ep'(x)|:x\in[c_\ep,d_\ep]\strut\right\}=0,
\nonumber
\end{equation}
and in particular $\theta_\ep\to 0$ in $C^1([a,b])$.

\paragraph{\textmd{\textit{Convergence of $\{\rho_\ep\}$}}}

We observe that $z_\ep+\theta_\ep$ has the same boundary conditions of $w_\ep$, both for the function and for the derivative, and therefore from the minimality of $w_\ep$ it follows that
\begin{equation}
    \FF_\ep((a_\ep,b_\ep),z_\ep+\theta_\ep+\rho_\ep)=
    \FF_\ep((a_\ep,b_\ep),w_\ep)\leq
    \FF_\ep((a_\ep,b_\ep),z_\ep+\theta_\ep).
\nonumber
\end{equation}

When we expand all the terms, this inequality turns out to be equivalent to
\begin{multline*}
    \tau_\ep^2\int_{a_\ep}^{b_\ep}\left[
    (\rho_\ep'')^2+2z_\ep''\rho_\ep''+2\theta_\ep''\rho_\ep''
    \right]dx+
    \int_{a_\ep}^{b_\ep}\left[
    \varphi_\ep(z_\ep'+\theta_\ep'+\rho_\ep')-\varphi_\ep(z_\ep'+\theta_\ep')
    \right]dx
    \\
    +\eta_\ep^2\int_{a_\ep}^{b_\ep}\left[
    \rho_\ep^2+2z_\ep\rho_\ep+2\theta_\ep\rho_\ep
    \right]dx-
    2\int_{a_\ep}^{b_\ep}h_\ep\rho_\ep\,dx\leq 0,
\end{multline*}
which in turn, by adding and subtracting equal terms, can be written in the form
\begin{equation}
    L_1+L_2+L_3+L_4\leq 0,
\nonumber
\end{equation}
where
\begin{gather*}
    L_1:=\tau_\ep^2\int_{a_\ep}^{b_\ep}(\rho_\ep'')^2\,dx+
    \frac{1}{2}\int_{a_\ep}^{b_\ep}(\rho_\ep')^2\,dx+
    \eta_\ep^2\int_{a_\ep}^{b_\ep}\rho_\ep^2\,dx+
    2\tau_\ep^2\int_{a_\ep}^{b_\ep}z_\ep''\rho_\ep''\,dx,
    \\
    L_2:=\int_{a_\ep}^{b_\ep}\left[
    \varphi_\ep(z_\ep'+\theta_\ep'+\rho_\ep')-
    \varphi_\ep(z_\ep'+\theta_\ep')-
    \varphi_\ep'(z_\ep'+\theta_\ep')\rho_\ep'
    -\frac{1}{2}(\rho_\ep')^2\right]dx
    \\
    L_3:=\int_{a_\ep}^{b_\ep}\varphi_\ep'(z_\ep')\rho_\ep'\,dx+
    2\int_{a_\ep}^{b_\ep}\left[\eta_\ep^2 z_\ep-h_\ep\right]\rho_\ep\,dx,
    \\
    L_4:=
    2\tau_\ep^2\int_{a_\ep}^{b_\ep}\theta_\ep''\rho_\ep''\,dx+
    \int_{a_\ep}^{b_\ep}\left[\varphi_\ep'(z_\ep'+\theta_\ep')-\varphi_\ep'(z_\ep')\right]\rho_\ep'\,dx+
    2\eta_\ep^2\int_{a_\ep}^{b_\ep}\theta_\ep\rho_\ep\,dx.
\end{gather*}

Now we observe that
\begin{itemize}
    \item $L_2\geq 0$ because of the convexity of $\varphi_\ep$ with the bound (\ref{hp:phin-convex}),

    \item  $L_3=0$ because it is the first variation of $\GG_\ep$, computed in the minimimizer $z_\ep$, with respect to the variation $\rho_\ep$,

    \item  $L_4=0$ because it is the first variation of $\HH_\ep$, computed in the minimimizer $\theta_\ep$, with respect to the variation $\rho_\ep$.
    
\end{itemize}

In conclusion we end up with
\begin{multline*}
    \quad
    \tau_\ep^2\int_{a_\ep}^{b_\ep}\rho_\ep''(x)^2\,dx+
    \frac{1}{2}\int_{a_\ep}^{b_\ep}\rho_\ep'(x)^2\,dx+
    \eta_\ep^2\int_{a_\ep}^{b_\ep}\rho_\ep(x)^2\,dx\\[0.5ex]
    \leq
    -2\tau_\ep^2\int_{a_\ep}^{b_\ep}z_\ep''(x)\rho_\ep''(x)\,dx\leq
    \frac{1}{2}\tau_\ep^2\int_{a_\ep}^{b_\ep}\rho_\ep''(x)^2\,dx+
    2\tau_\ep^2\int_{a_\ep}^{b_\ep}z_\ep''(x)^2\,dx,
    \quad
\end{multline*}
and hence
\begin{equation}
    \frac{1}{2}\tau_\ep^2\int_{a_\ep}^{b_\ep}\rho_\ep''(x)^2\,dx+
    \frac{1}{2}\int_{a_\ep}^{b_\ep}\rho_\ep'(x)^2\,dx+
    \eta_\ep^2\int_{a_\ep}^{b_\ep}\rho_\ep(x)^2\,dx\leq
    2\tau_\ep^2\int_{a_\ep}^{b_\ep}z_\ep''(x)^2\,dx.
\nonumber
\end{equation}

Since we know that $\{z_\ep\}$ is bounded in $C^2([a_\ep,b_\ep])$, we deduce that $\{\rho''_\ep\}$ is bounded in $L^2((a_\ep,b_\ep))$, and the norm of $\rho_\ep'$ in $L^2((a_\ep,b_\ep))$ tends to 0. Taking into account the boundary conditions (\ref{BC:theta_ep}), this is enough to conclude that the norm of $\rho_\ep$ in $C^1([a_\ep,b_\ep])$, and a fortiori in $C^1([a',b'])$, tends to 0.
\end{proof}


\setcounter{equation}{0}
\section{Vertical regions -- Proof of Theorems~\ref{thm:vertical-step} and~\ref{thm:vertical}}\label{sec:vertical}

In this section we prove Theorem~\ref{thm:vertical-step}, which in turn implies Theorem~\ref{thm:vertical}. We commence by defining
\begin{equation}
    w_\ep(y):=v_\ep(\ep^2 y).
\nonumber
\end{equation}

The key point of the proof lies in demonstrating that the transition of $v_\ep$, and hence also of $w_\ep$, from $-V$ to $V$ occurs within a single interval, which we term the ``big jump'', with a length approximately proportional to $\ep^2$ in the case of $v_\ep$, and hence of the order of a constant in the case of $w_\ep$. Within this interval, we establish uniform estimates for $w_\ep$, $w_\ep'$, and $w_\ep''$, enabling us to uniquely characterize the limit. 

To begin with, we observe that $w_{\ep}(y)$ is defined for every $|y|<L/\ep^{2}$, and with a change of variable in the integral we obtain that
\begin{equation}
\RPM_{\ep}((-L,L),v_{\ep})=\RPMV_{\ep}\left((-L/\ep^{2},L/\ep^{2}),w_{\ep}\right),
\label{eqn:ABGvn-ABGwn}
\end{equation}
where $\RPMV_\ep$ denotes the rescaled Perona-Malik functional for vertical regions, defined by
\begin{equation*}
\RPMV_\ep(\Omega,w):= \int_\Omega \left\{ w''(y)^2 + \frac{1}{|\log\ep|}\log\left(1+\frac{w'(y)^2}{\ep^4}\right)\right\}dy.
\end{equation*}

Assumption (\ref{hp:vn2jump}) implies that $v_\ep(y)\to V$ uniformly in $[a,L]$, and hence $w_\ep(y)\to V$ uniformly in $[a/\ep^2,L/\ep^2]$, for every $a\in(0,L)$. Similarly,  $w_\ep(y)\to -V$ uniformly in $[-L/\ep^2,-a/\ep^2]$, for every $a\in(0,L)$. Moreover, assumption (\ref{hp:ABGvn}) implies that the right-hand side of (\ref{eqn:ABGvn-ABGwn}) is bounded from above by some constant $\Gamma$, at least when $\ep\in(0,\ep_0)$ for some $\ep_0$ small enough, and therefore any interval of length at least $\Gamma|\log\ep|/\log 2$ contains at least one point $y$ where $|w_\ep'(y)|\leq\ep^2$. 

As a consequence, there exist
\begin{equation}
    a_\ep\in\left[-\frac{L}{\ep^2},-\frac{L}{\ep^2}+\frac{\Gamma}{\log 2}|\log\ep|\right]
    \qquad\text{and}\qquad
    b_\ep\in\left[\frac{L}{\ep^2}-\frac{\Gamma}{\log 2}|\log\ep|,\frac{L}{\ep^2}\right]
\nonumber
\end{equation}
such that 
\begin{equation}
    w_\ep(a_\ep)\to -V,
    \qquad
    w_\ep(b_\ep)\to V,
    \qquad
    |w_\ep'(a_\ep)|\leq\ep^2,
    \qquad
    |w_\ep'(b_\ep)|\leq\ep^2.
    \label{defn:aep-bep}
\end{equation}

In particular, from the first two relations, and the strict convergence assumption (\ref{hp:ABGvn}), we obtain that
\begin{equation}
    2V\leq
    \liminf_{\ep\to 0^+}\int_{a_\ep}^{b_\ep}|w_\ep'(y)|\,dy\leq
    \lim_{\ep\to 0^+}\int_{-L/\ep^2}^{L/\ep^2}|w_\ep'(y)|\,dy=
    \lim_{\ep\to 0^+}\int_{-L}^{L}|v_\ep'(x)|\,dx=
    2V,
\nonumber
\end{equation}
which implies that
\begin{equation}
    \lim_{\ep\to 0^+}\int_{a_\ep}^{b_\ep}|w_\ep'(y)|\,dy=2V.
    \label{strict-conv}
\end{equation}

\paragraph{\textmd{\textit{Identification of the ``big jump''}}}

For every $\ep\in(0,\ep_0)$ we write $(a_\ep,b_\ep)$ as the disjoint union of the sets
\begin{gather*}
\mathcal{A}_{\ep}:=\left\{y\in(a_\ep,b_\ep):|w_{\ep}'(y)|> \frac{1}{|\log\ep|^2}\right\},
\qquad
\mathcal{C}_{\ep}:=\left\{y\in(a_\ep,b_\ep):|w_{\ep}'(y)|< \ep^{3/2}\right\},\\
\mathcal{B}_{\ep}:=\left\{y\in(a_\ep,b_\ep): \ep^{3/2}\leq |w_{\ep}'(y)|\leq\frac{1}{|\log\ep|^2}\right\},
\end{gather*}
and we further write the open set $\mathcal{A}_\ep$ as the union of its connected components
\begin{equation*}
\mathcal{A}_{\ep}=\bigcup_{i\in I_{\ep}}(c_{\ep,i},d_{\ep,i}),
\end{equation*}
where $I_{\ep}$ is a finite or countable set of indices. We also set
\begin{equation}
\Delta_{\ep,i}:=
\int_{c_{\ep,i}}^{d_{\ep,i}}|w_{\ep}'(y)|\,dy
\qquad\qquad\text{and}\qquad\qquad
\widehat{\Delta}_{\ep,i}:=
\Delta_{\ep,i}-\frac{d_{\ep,i}-c_{\ep,i}}{|\log\ep|^2}.
\nonumber
\end{equation}

We claim that for every $\ep\in(0,\ep_0)$ there exists an index $i(\ep)\in I_{\ep}$ such that
\begin{equation}
\lim_{\ep\to 0^+}\Delta_{\ep,i(\ep)}=
2V.
\label{th:big-jump}
\end{equation}

In words, this means that asymptotically the whole total variation of $w_{\ep}$ is realized in a single special interval $(c_{\ep,i(\ep)},d_{\ep,i(\ep)})$. We divide the proof into five steps.
\begin{itemize}

\item  Step 1. We show that 
\begin{equation}
\limsup_{\ep\to 0^+}|\mathcal{A}_{\ep}|<+\infty
\label{th:A-bounded}
\end{equation}
and
\begin{equation}
\lim_{\ep\to 0^+}\int_{\mathcal{C}_\ep}|w_\ep'(y)|\,dy=
\lim_{\ep\to 0^+}\int_{\mathcal{B}_\ep}|w_\ep'(y)|\,dy=
\lim_{\ep\to 0^+}\frac{|\mathcal{A}_{\ep}|}{|\log\ep|^2}=
0.
\label{th:TV-BC-bounded}
\end{equation}

To this end, it is enough to apply Lemma~\ref{lemma:stima-ABGabc} with
\begin{equation}
    \lambda:=1,
    \qquad
    \mu:=\frac{1}{|\log\ep|},
    \qquad
    \nu:=\frac{1}{\ep^4},
    \qquad
    D:=\ep^{3/2},
    \qquad
    M:=\frac{1}{|\log\ep|^2}.
    \label{choices:vertical}
\end{equation}

With these choices, (\ref{th:small-der}), (\ref{th:interm-der}), and (\ref{th:est-AM}) read as
\begin{gather*}
    \left(\int_{\mathcal{C}_\ep}|w_\ep'(y)|\,dy\right)^2\leq
    \RPMV_\ep((a_\ep,b_\ep),w_\ep)\cdot\ep^3(b_\ep-a_\ep)\cdot\frac{|\log\ep|}{\log(1+\ep^{-1})},
    \\[0.5ex]
    \int_{\mathcal{B}_\ep}|w_\ep'(y)|\,dy\leq
    \RPMV_\ep((a_\ep,b_\ep),w_\ep)\cdot
    \frac{|\log\ep|}{\log(1+\ep^{-1})}\cdot
    \frac{1}{|\log\ep|^2},
    \\[0.5ex]
    |\mathcal{A}_{\ep}|\leq
    \RPMV_\ep((a_\ep,b_\ep),w_\ep)\cdot
    \frac{|\log\ep|}{\log(1+\ep^{-4}|\log\ep|^{-4})}.
\end{gather*}

Recalling that $\RPMV_\ep((a_\ep,b_\ep),w_\ep)$ is bounded from above, and $b_\ep-a_\ep\leq 2L/\ep^2$, when we let $\ep\to 0^+$ we obtain exactly both (\ref{th:A-bounded}) and (\ref{th:TV-BC-bounded}).

\item  Step 2. We show that
\begin{equation}
\lim_{\ep\to 0^+}\sum_{i\in I_{\ep}}\Delta_{\ep,i}=
\lim_{\ep\to 0^+}\sum_{i\in I_{\ep}}\widehat{\Delta}_{\ep,i}=
2V.
\label{th:int-AnC}
\end{equation}

Indeed we observe that
\begin{equation}
    \sum_{i\in I_{\ep}}\Delta_{\ep,i}=
    \int_{\mathcal{A}_\ep}|w_\ep'(y)|\,dy=
    \int_{a_\ep}^{b_\ep}|w_\ep'(y)|\,dy-\int_{\mathcal{B}_\ep}|w_\ep'(y)|\,dy-\int_{\mathcal{C}_\ep}|w_\ep'(y)|\,dy,
\nonumber
\end{equation}
and
\begin{equation}
    \sum_{i\in I_{\ep}}\widehat{\Delta}_{\ep,i}=
    \sum_{i\in I_{\ep}}\Delta_{\ep,i}-
    \frac{|\mathcal{A}_{\ep}|}{|\log\ep|^2}.
\nonumber
\end{equation}

At this point (\ref{th:int-AnC}) follows from (\ref{strict-conv}) and (\ref{th:TV-BC-bounded}).

\item  Step 3. We show that
\begin{equation}
\limsup_{\ep\to 0^+}\sum_{i\in I_{\ep}}\left(\widehat{\Delta}_{\ep,i}\right)^{1/2}\leq
\sqrt{2V}.
\label{th:limsup-hat}
\end{equation}

Indeed, for every $i\in I_\ep$ we apply (\ref{th:high-deriv-cc}) of Lemma~\ref{lemma:stima-ABGabc}, whose assumptions are satisfied with the choices (\ref{choices:vertical}) because of the two inequalities in (\ref{defn:aep-bep}), and we obtain that
$$\RPMV_{\ep}((c_{\ep,i},d_{\ep,i}),w_{\ep})\geq 
4\sqrt{\frac{2}{3}}\left\{\frac{1}{|\log\ep|} \log\left(1+\frac{1}{\ep^4|\log\ep|^4}\right)\right\}^{3/4} \left(\widehat{\Delta}_{\ep,i}\right)^{1/2}.$$

Summing over all indices $i\in I_\ep$ we deduce that
\begin{equation*}
\RPMV_{\ep}((a_{\ep},b_{\ep}),w_{\ep})\geq
4\sqrt{\frac{2}{3}}\left\{\frac{1}{|\log\ep|} \log\left(1+\frac{1}{\ep^4|\log\ep|^4}\right)\right\}^{3/4} \sum_{i\in I_{\ep}}\left(\widehat{\Delta}_{\ep,i}\right)^{1/2},
\end{equation*}
and, thanks to (\ref{eqn:ABGvn-ABGwn}) and (\ref{hp:ABGvn}), we conclude by letting $\ep\to 0^+$.

\item  Step 4. We show that 
\begin{equation}
\lim_{\ep\to 0^+}\sum_{i\in I_{\ep}}\left(\widehat{\Delta}_{\ep,i}\right)^{1/2}=
\sqrt{2V}.
\label{th:lim-hat}
\end{equation}

Indeed, by combining (\ref{th:limsup-hat}), the subadditivity of the square root, and (\ref{th:int-AnC}), we obtain the chain of inequalities
\begin{multline*}
\qquad
\sqrt{2V}\geq
\limsup_{\ep\to 0^+}\sum_{i\in I_{\ep}}\left(\widehat{\Delta}_{\ep,i}\right)^{1/2}\geq
\liminf_{\ep\to 0^+}\sum_{i\in I_{\ep}}\left(\widehat{\Delta}_{\ep,i}\right)^{1/2}\geq
\\[1ex]
\geq\liminf_{\ep\to 0^+}\left(\sum_{i\in I_{\ep}}\widehat{\Delta}_{\ep,i}\right)^{1/2}=
\left(\liminf_{\ep\to 0^+}\sum_{i\in I_{\ep}}\widehat{\Delta}_{\ep,i}\right)^{1/2}=
\sqrt{2V}.
\qquad
\end{multline*}

This means that all inequalities are actually equalities, which proves (\ref{th:lim-hat}).

\item  Step 5. We prove (\ref{th:big-jump}).

To this end, we apply Lemma~\ref{lemma:MRS} with $I:=I_{\ep}$ and $f(i):=\widehat{\Delta}_{\ep,i}$, and we deduce that
\begin{equation*}
0\leq
\left(\sum_{i\in I_{\ep}}\widehat{\Delta}_{\ep,i}-\max_{i\in I_{\ep}}\widehat{\Delta}_{\ep,i}\right)^{1/2}\leq
3\sum_{i\in I_{\ep}}\left(\widehat{\Delta}_{\ep,i}\right)^{1/2}-
3\left(\sum_{i\in I_{\ep}}\widehat{\Delta}_{\ep,i}\right)^{1/2}.
\end{equation*}

The right-hand side tends to~0 because of (\ref{th:lim-hat}) and (\ref{th:int-AnC}), and hence there exists $i(\ep)\in I_{\ep}$ such that
\begin{equation*}
\lim_{\ep\to 0^+}\widehat{\Delta}_{\ep,i(\ep)}=
\lim_{\ep\to 0^+}\sum_{i\in I_{\ep}}\widehat{\Delta}_{\ep,i}=
2V.
\end{equation*}

At this point the conclusion follows from the last equality in (\ref{th:TV-BC-bounded}) because 
\begin{equation}
    \widehat{\Delta}_{\ep,i(\ep)}\leq
    \Delta_{\ep,i(\ep)}=
    \widehat{\Delta}_{\ep,i(\ep)}+\frac{d_{\ep,i(\ep)}-c_{\ep,i(\ep)}}{|\log\ep|^2}\leq
    \widehat{\Delta}_{\ep,i(\ep)}+\frac{|\mathcal{A}_{\ep}|}{|\log\ep|^2}.
\nonumber
\end{equation}

\end{itemize}

\paragraph{\textmd{\textit{Uniform estimates in the special interval}}}

For the sake of shortness let us set $c_{\ep}:=c_{\ep,i(\ep)}$ and $d_{\ep}:=d_{\ep,i(\ep)}$. We recall that the two inequalities in (\ref{defn:aep-bep}) imply that $c_\ep$ and $d_\ep$ are different, respectively, from $a_\ep$ and $b_\ep$, and therefore
\begin{equation}
|w_{\ep}'(c_{\ep})|=|w_{\ep}'(d_{\ep})|=\frac{1}{|\log\ep|^2}\to 0.
\label{lim:wn'-an}
\end{equation}

We claim that the following three facts hold true.
\begin{itemize}

\item  The family $\{w_{\ep}(y)\}$ satisfies
\begin{equation}
\lim_{\ep\to 0^+}\left(
\max_{y\in[a_\ep,c_{\ep}]}|w_{\ep}(y)+V|+
\max_{y\in[d_{\ep},b_\ep]}|w_{\ep}(y)-V|
\right)=0,
\label{th:wn-uc}
\end{equation}
and in particular
\begin{equation}
w_{\ep}(c_{\ep})\to -V
\qquad\text{and}\qquad
w_{\ep}(d_{\ep})\to V.
\label{lim:wn-an}
\end{equation}

\item  When $\ep$ is small enough it turns out that 
\begin{equation*}
-|\mathcal{A}_{\ep}|<c_{\ep}<0<d_{\ep}<|\mathcal{A}_{\ep}|,
\end{equation*}
and due to (\ref{th:A-bounded}) this implies that the families $\{c_{\ep}\}$ and $\{d_{\ep}\}$ are bounded.

\item  For every bounded interval $(a,b)\subseteq\re$ it turns out that
\begin{equation}
\{w_{\ep}\}
\text{ is bounded in }H^{2}((a,b)),
\label{th:wn-bounded-H2}
\end{equation}
of course with a bound that depends on $(a,b)$.

\end{itemize}

Let us prove these claims. To begin with, from (\ref{strict-conv}) and (\ref{th:big-jump}) we know that
\begin{equation*}
\lim_{\ep\to 0^+}
\int_{a_\ep}^{c_{\ep}}|w_{\ep}'(y)|\,dy+\int_{d_{\ep}}^{b_\ep}|w_{\ep}'(y)|\,dy=0.
\end{equation*}

Taking the first two properties in (\ref{defn:aep-bep}) into account, this is enough to establish (\ref{th:wn-uc}), from which we deduce in particular that
\begin{equation*}
w_{\ep}(y)\leq-\frac{V}{2}<0
\qquad
\forall y\in[a_\ep,c_{\ep}]
\end{equation*}
when $\ep$ is small enough. Since $w_{\ep}(0)=v_\ep(0)=0$ for every $\ep\in(0,1)$, we conclude that $c_{\ep}<0$, and also
\begin{equation*}
d_{\ep}=c_{\ep}+(d_{\ep}-c_{\ep})\leq
c_{\ep}+|\mathcal{A}_{\ep}|<|\mathcal{A}_{\ep}|.
\end{equation*}

In a symmetric way we obtain that $d_{\ep}>0$ and $c_{\ep}>-|\mathcal{A}_{\ep}|$.

Finally, the bound on $\RPMV_{\ep}((a_\ep,b_\ep),w_{\ep})$ yields immediately a uniform bound on the norm of $w_{\ep}''$ in $L^{2}((a_\ep,b_\ep))$. Together with the pointwise bounds coming from (\ref{lim:wn-an}) and (\ref{lim:wn'-an}), and the boundedness of $\{c_{\ep}\}$ and $\{d_{\ep}\}$, this implies (\ref{th:wn-bounded-H2}).

\paragraph{\textmd{\textit{Passing to the limit}}}

Exploiting again that the families $\{c_{\ep}\}$ and $\{d_{\ep}\}$ are bounded, up to subsequences (not relabeled) we can assume that $c_{\ep}\to c_{0}$ and $d_{\ep}\to d_{0}$. 

Now let us fix any interval $(a,b)\supseteq(c_{0},d_{0})$. From (\ref{th:wn-bounded-H2}) we can assume that there exists $w_{0}\in H^{2}((a,b))$ such that, up to further subsequences (again not relabeled),
\begin{equation}
w_{\ep}\to w_{0}
\quad\text{and}\quad
w_{\ep}'\to w_{0}'
\quad\text{uniformly in }[a,b],
\nonumber
\end{equation}
and
\begin{equation}
w_{\ep}''\rightharpoonup w_{0}''
\quad\text{weakly in }L^{2}((a,b)).
\nonumber
\end{equation}

We claim that $w_{0}$ is the cubic connection. To begin with, we observe that  $w_0(0)=0$ because by definition $w_\ep(0)=v_\ep(0)=0$ for every $\ep\in(0,1)$, while from (\ref{lim:wn-an}) and (\ref{lim:wn'-an}) we obtain that
\begin{equation}
w_{0}(c_{0})=-V,
\qquad
w_{0}(d_{0})=V,
\qquad
w_{0}'(c_{0})=w_{0}'(d_{0})=0.
\label{BC:w-infty}
\end{equation}

Due to these boundary conditions, from Lemma~\ref{lemma:ABG} we deduce that
\begin{equation*}
\int_{c_{0}}^{d_{0}}w_{0}''(y)^{2}\,dy\geq
\frac{12(2V)^{2}}{(d_{0}-c_{0})^{3}}.
\end{equation*}

Now we consider the chain of inequalities
\begin{eqnarray}
\frac{16}{\sqrt{3}}\sqrt{2V} & \geq &
\limsup_{\ep\to 0^+}\RPMV_{\ep}((a_\ep,b_\ep),w_{\ep})
\nonumber\\
& \geq & 
\limsup_{\ep\to 0^+}\left\{
\int_{c_{\ep}}^{d_{\ep}}\frac{1}{|\log\ep|}\log\left(1+\frac{w_{\ep}'(y)^{2}}{\ep^{4}}\right)\,dy+
\int_{a}^{b}w_{\ep}''(y)^{2}\,dy
\right\}
\nonumber\\
& \geq & 
\limsup_{\ep\to 0^+}\left\{\frac{1}{|\log\ep|}\log\left(\frac{1}{\ep ^4 |\log\ep|^4}\right)(d_{\ep}-c_{\ep})+
\int_{a}^{b}w_{\ep}''(y)^{2}\,dy
\right\}
\nonumber\\
& = & 
4(d_{0}-c_{0})+
\limsup_{\ep\to 0^+}\int_{a}^{b}w_{\ep}''(y)^{2}\,dy
\nonumber\\
& \geq & 
4(d_{0}-c_{0})+
\int_{a}^{b}w_{0}''(y)^{2}\,dy
\label{eq:strong}
\\
& \geq & 
4(d_{0}-c_{0})+
\int_{c_{0}}^{d_{0}}w_{0}''(y)^{2}\,dy
\label{eq:0-outside}
\\
& \geq & 
4(d_{0}-c_{0})+
12\frac{(2V)^{2}}{(d_{0}-c_{0})^{3}}
\label{eq:cubic}
\\
& \geq & \frac{16}{\sqrt{3}}\sqrt{2V}.
\label{eq:lenght}
\end{eqnarray}

Since the first and last term coincide, all inequalities are actually equalities, and each of them gives us some piece of information.
\begin{itemize}

\item  The equality in (\ref{eq:lenght}) implies that $d_{0}-c_{0}=2\Lambda$, where $\Lambda$ is defined by (\ref{defn:Lambda}).

\item  The equality in (\ref{eq:cubic}) implies that $w_{0}$ is the unique polynomial of degree three with the boundary data (\ref{BC:w-infty}). Combining with the previous point and the fact that $w_0(0)=0$, we conclude that $w_{0}$ coincides with the  $(\Lambda,V)$-cubic connection in the interval $(c_{0},d_{0})=(-\Lambda,\Lambda)$.

\item  The equality in (\ref{eq:0-outside}) implies that $w_{0}''(y)$ vanishes outside $(c_{0},d_{0})$. Since $w_{0}'(c_{0})=w_{0}'(d_{0})=0$, this is enough to conclude that $w_{0}$ is constant both in $(a,c_{0})$ and in $(d_{0},b)$, and hence it coincides with the $(\Lambda,V)$-cubic connection in $(a,b)$.

\item  The equality in (\ref{eq:strong}) implies that actually $w_{\ep}''\to w_{0}''$ strongly in $L^{2}((a,b))$.

\end{itemize}

Finally, we observe that the previous steps characterize in a unique way the possible limits of subsequences, and this is enough to conclude the convergence of the whole family $\{w_\ep\}$ to the  $(\Lambda,V)$-cubic connection. 
\qed


\setcounter{equation}{0}
\section{Horizontal regions -- Proof of Theorem~\ref{thm:horizontal}}\label{sec:horiz}

Let us consider the family of functions
\begin{equation}
    w_\ep(y):=\frac{v_\ep(y)}{\omep^2}
    \qquad
    \forall y\in(-H,H).
    \label{defn:wep-hor}
\end{equation}

In the first part of the proof we show that the family $\{w_\ep\}$ is bounded in $C^0([a,b])$ for every interval $[a,b]\subseteq(-H,H)$, while in the second part we show that the same family is relatively compact in $C^1([a,b])$. We also show that all the limit points are solutions to a limit minimum problem, whose minimizers are the cubic polynomials (\ref{th:hor-structure}).

\subsection{Boundedness of the uniform norm}\label{sec:horiz-1}

In this part of the proof we show that
\begin{equation}
    \limsup_{\ep\to 0^+}\max\{|w_\ep(y)|:x\in[a,b]\}<+\infty.
    \label{th:wep-bounded}
\end{equation}

In a nutshell, the idea is the following. With a change of variable in the integrals we obtain that
$$\PMF_{\ep}(\beta,f,(0,1),u_{\ep})=\omep^3 \RPMF_{\ep}(\beta,g_{\ep},(-x_{\ep}/\omep,(1-x_{\ep})/\omep),v_{\ep}),$$
where 
\begin{equation}
    g_{\ep}(y):= \frac{f(x_{\ep}+\omep y)-u_{\ep}(x_{\ep})}{\omep}.
\label{defn:gep}
\end{equation}

Since $u_\ep$ is a minimizer of $\PMF_\ep$, we deduce that $v_\ep$ is a minimizer of $\RPMF_\ep$ with respect to its own boundary conditions in every admissible interval. In particular, if we consider any point $c\in(H,2H)$, we know that
\begin{equation}
    \RPMF_\ep(\beta,g_\ep,(a,c),v_\ep)\leq
    \RPMF_\ep(\beta,g_\ep,(a,c),v_\ep+\phi)
    \qquad
    \forall\phi\in H^2_0((a,c))
\nonumber
\end{equation}
when $\ep$ is small enough. In the interval $(a,c)$ the family $\{v_\ep(y)\}$ converges to a piecewise constant function with a jump from 0 to $2V$ in $y=H$.  If condition (\ref{th:wep-bounded}) does not hold, then $v_\ep$ exhibits a significant oscillation in $[a,b]$, namely ``before the jump''. This substantial oscillation incurs a cost in terms of rescaled Perona-Malik energy, which might initially seem offset by the need for a ``shorter jump'' in $(b,c)$. However, we show that eliminating this oscillation, even at the expense of a larger jump in $(b,c)$, yields a better competitor, which contradicts the minimality of $v_\ep$ in $(a,c)$.

In order to make this idea more quantitative, we consider a larger interval $(a',b')$ with $-H<a'<a<b<b'<H$, and an interval $(c,c')$ with $H<c<c'<2H$. From assumption (\ref{hp:horiz}) and the properties of strict convergence we know that
\begin{equation}
    v_\ep(y)\to 0
    \qquad
    \text{uniformly in }[a',b'],
    \label{th:vep-to-0}
\end{equation}
\begin{equation}
    v_\ep(y)\to 2V
    \qquad
    \text{uniformly in }[c,c'],
    \label{th:vep-to-2V}
\end{equation}
while from statement~(1) of Theorem~\ref{thmbibl:add-on} we know that
\begin{equation}
    \limsup_{\ep\to 0^+}\RPM_\ep((a',c'),v_\ep)<+\infty.
    \label{th:RPM-bounded}
\end{equation}

Now we consider the interval $[a',a]$, and we observe that
\begin{equation}
    \frac{1}{\omep^2}\int_{a'}^a \log\left(1+v_\ep'(y)^2\right)\,dy\leq
    \RPM_\ep((a',a),v_\ep).
\nonumber
\end{equation}

The right-hand side is bounded for $\ep$ small enough because of (\ref{th:RPM-bounded}), and hence there exists a family $\{a_\ep\}\subseteq[a',a]$ such that $|v_\ep'(a_\ep)|=O(\omep)$. Arguing in the same way in the intervals $[b,b']$ and $[c,c']$, and taking (\ref{th:vep-to-0}) and (\ref{th:vep-to-2V}) into account, we conclude that there exist three families $\{a_\ep\}\subseteq[a',a]$, $\{b_\ep\}\subseteq[b,b']$ and $\{c_\ep\}\subseteq[c,c']$ such that
\begin{equation}
    v_\ep(a_\ep)\to 0,
    \qquad\qquad
    v_\ep(b_\ep)\to 0,
    \qquad\qquad
    v_\ep(c_\ep)\to 2V,
\label{th:lim-abc}
\end{equation}
and for which there exists a positive real number $\gamma_1$ such that
\begin{equation}
    |v_\ep'(a_\ep)|\leq\gamma_1\omep,
    \qquad\quad
    |v_\ep'(b_\ep)|\leq\gamma_1\omep,
    \qquad\quad
    |v_\ep'(c_\ep)|\leq\gamma_1\omep
\label{est:vep'-abc}
\end{equation}
for every small enough $\ep$.

We observe also that the functions $g_{\ep}(y)$ defined by (\ref{defn:gep}) are uniformly bounded on bounded sets because of statement~(2) of Theorem~\ref{thmbibl:add-on}.

Now we set
\begin{equation}
    m_\ep:=\max\{|v_\ep(y)|:x\in[a_\ep,b_\ep]\},
\nonumber
\end{equation}
and we observe that $m_\ep\to 0$ because of the uniform convergence (\ref{th:vep-to-0}). We claim that actually
\begin{equation}
    \limsup_{\ep\to 0^+}\frac{m_\ep}{\omep^2}<+\infty,
\label{th:wep-bounded-bis}
\end{equation}
which implies (\ref{th:wep-bounded}). The proof of the claim is derived from the following two estimates (all the constants $\gamma_i$ introduced hereafter are independent of $\ep$, and all the estimates are understood to be true when $\ep$ is small enough).
\begin{itemize}

\item  (Estimate from below). There exist two positive real numbers $\gamma_2$ and $\gamma_3$ such that
\begin{equation}
\RPM_{\ep}((a_{\ep},b_{\ep}),v_{\ep})\geq
\min\left\{\gamma_2\sqrt{m_{\ep}},\gamma_3\frac{m_{\ep}^{2}}{\omep^{2}}\right\}.
\label{th:RPM>=}
\end{equation}

\item  (Estimate from above). There exist two real numbers $\gamma_4$ and $\gamma_5$ such that
\begin{eqnarray}
\RPMF_{\ep}(\beta,g_{\ep},(a_{\ep},c_{\ep}),v_{\ep})& \leq &
\RPMF_{\ep}(\beta,g_{\ep},(b_{\ep},c_{\ep}),v_{\ep})+\beta\int_{a_{\ep}}^{b_{\ep}}g_{\ep}(y)^{2}\,dy
\nonumber   \\
& & \mbox{}+
\gamma_4 m_{\ep}+\gamma_5\omep^{2}.
\label{th:RPMF<=}
\end{eqnarray}

\end{itemize}

These two claims are enough to conclude. Indeed, from the definition of $m_\ep$, and the uniform boundedness of $g_\ep(y)$, it follows that
\begin{equation}
    \int_{a_\ep}^{b_\ep}(v_\ep(y)-g_\ep(y))^2\,dy\geq
    \int_{a_\ep}^{b_\ep}g_\ep(y)^2\,dy-\gamma_6 m_\ep.
\nonumber
\end{equation}

Thanks to this estimate, by combining (\ref{th:RPM>=}) and (\ref{th:RPMF<=}) we obtain the chain of inequalities 
\begin{eqnarray*}
\lefteqn{
\RPMF_{\ep}(\beta,g_{\ep},(b_{\ep},c_{\ep}),v_{\ep})+\beta\int_{a_{\ep}}^{b_{\ep}}g_{\ep}(y)^{2}\,dy+
\gamma_4 m_{\ep}+\gamma_5\omep^{2}}
\\[1ex]
& \geq & \RPMF_{\ep}(\beta,g_{\ep},(a_{\ep},c_{\ep}),v_{\ep})
\\[1ex]
& = & \RPM_{\ep}((a_{\ep},b_{\ep}),v_{\ep})
+\beta\int_{a_{\ep}}^{b_{\ep}}(v_{\ep}(y)-g_{\ep}(y))^{2}\,dy
+\RPMF_{\ep}(\beta,g_{\ep},(b_{\ep},c_{\ep}),v_{\ep})
\\[1ex]
& \geq & 
\min\left\{\gamma_2\sqrt{m_{\ep}},\gamma_3\frac{m_{\ep}^{2}}{\omep^{2}}\right\}
+\beta\int_{a_{\ep}}^{b_{\ep}}g_{\ep}(y)^{2}\,dy
-\beta\gamma_6 m_{\ep}+
\RPMF_{\ep}(\beta,g_{\ep},(b_{\ep},c_{\ep}),v_{\ep}),
\end{eqnarray*}
from which we deduce that
\begin{equation}
\min\left\{\gamma_2\sqrt{m_{\ep}},\gamma_3\frac{m_{\ep}^{2}}{\omep^{2}}\right\}\leq
(\gamma_4+\beta\gamma_6) m_{\ep}+\gamma_5\omep^{2},
\label{est:Mk-O}
\end{equation}
which in turn implies (\ref{th:wep-bounded-bis}). Indeed, if this is not the case, then there exists a subsequence (not relabeled) such that $m_{\ep}/\omep^{2}\to +\infty$ as $\ep\to 0^+$. Since eventually $m_{\ep}\neq 0$ along this subsequence, we can divide (\ref{est:Mk-O}) by $m_{\ep}$ and obtain that
\begin{equation*}
\min\left\{\frac{\gamma_2}{\sqrt{m_{\ep}}},\gamma_3\frac{m_{\ep}}{\omep^{2}}\right\}\leq
\gamma_4+\beta\gamma_6+\gamma_5\frac{\omep^{2}}{m_{\ep}},
\end{equation*}
which is absurd because the left-hand side tends to $+\infty$ when $\ep\to 0^+$, while the right-hand side remains bounded.

The rest of the proof is devoted to showing (\ref{th:RPM>=}) and (\ref{th:RPMF<=}).

\paragraph{\textmd{\textit{Estimate from below}}}

Since $v_\ep(0)=0$, the total variation of $v_\ep$ in $(a_\ep,b_\ep)$ is greater than or equal to $m_\ep$. Consequently, (\ref{th:RPM>=}) follows from Lemma~\ref{lemma:RPM>=} applied to the function $v_\ep$ in the interval $(a_\ep,b_\ep)$. A careful inspection of the right-hand sides of (\ref{th:RPM-below}) and (\ref{hp:ep-small}) reveals that in this point it is important that the length of $(a_\ep,b_\ep)$ is bounded from below by a positive constant.

\paragraph{\textmd{\textit{Estimate from above}}}

In this paragraph we prove (\ref{th:RPMF<=}). Since $v_{\ep}$ minimizes the rescaled Perona-Malik functional in the interval $(a_\ep,c_\ep)$ with respect to its boundary conditions, it is enough to exhibit a competitor with the same boundary conditions whose energy is bounded from above by the right-hand side of (\ref{th:RPMF<=}). To this end, we consider the function $\overline{v}_{\ep}:[a_{\ep},c_{\ep}]\to\re$ defined by
\begin{equation}
    \overline{v}_{\ep}(y):=
    \begin{cases}
    v_{\ep}(a_{\ep}) & \text{if }y\in[a_{\ep},b_{\ep}],
    \\
    v_{\ep}(c_{\ep})+H_{\ep}(v_{\ep}(y)-v_{\ep}(c_{\ep})) & \text{if }y\in[b_{\ep},c_{\ep}],
    \end{cases}
\nonumber
\end{equation}
where
\begin{equation}
    H_{\ep}:=\frac{v_{\ep}(c_{\ep})-v_{\ep}(a_{\ep})}{v_{\ep}(c_{\ep})-v_{\ep}(b_{\ep})}.
\nonumber
\end{equation}

In words, $\overline{v}_{\ep}(y)$ is obtained by gluing the function that is constantly equal to $v_{\ep}(a_{\ep})$ in $[a_{\ep},b_{\ep}]$, and the homotety of $v_{\ep}(y)$ whose graph passes through the points $(b_{\ep},v_{\ep}(a_{\ep}))$ and $(c_{\ep},v_{\ep}(c_{\ep}))$. 

We observe that $\overline{v}_{\ep}(y)$ coincides with $v_{\ep}(y)$ for $y\in\{a_{\ep},c_{\ep}\}$, but nevertheless $\overline{v}_{\ep}(y)$ is not yet an admissible competitor, both because $\overline{v}_{\ep}'(y)$ does not necessarily coincide with $v_{\ep}'(y)$ for $y\in\{a_{\ep},c_{\ep}\}$, and because the junction in $b_{\ep}$ is just continuous and does not provide necessarily a function in $H^2((a_{\ep},c_{\ep}))$.

In order to overcome these problems, we introduce a modified version of $\overline{v}_{\ep}(y)$. We consider the function $\hatv_{\ep}:[a_{\ep},c_{\ep}]\to\re$ such that
\begin{itemize}
    \item in the interval $[a_{\ep},b_{\ep}]$ it is a minimizer of $\RPMF_{\ep}(\beta,g_{\ep},(a_{\ep},b_{\ep}),v)$ with boundary conditions
    \begin{equation}
        \hatv_{\ep}(a_{\ep})=v_{\ep}(a_{\ep}),
        \qquad
        \hatv_{\ep}'(a_{\ep})=v_{\ep}'(a_{\ep}),
        \qquad
        \hatv_{\ep}(b_{\ep})=v_{\ep}(a_{\ep}),
        \qquad
        \hatv_{\ep}'(b_{\ep})=0,
        \nonumber
    \end{equation}
    
    \item in the interval $[b_{\ep},c_{\ep}]$ it is a minimizer of $\RPMF_{\ep}(\beta,g_{\ep},(b_{\ep},c_{\ep}),v)$ with boundary conditions
    \begin{equation}
        \hatv_{\ep}(b_{\ep})=v_{\ep}(a_{\ep}),
        \qquad
        \hatv_{\ep}'(b_{\ep})=0,
        \qquad
        \hatv_{\ep}(c_{\ep})=v_{\ep}(c_{\ep}),
        \qquad
        \hatv_{\ep}'(c_{\ep})=v_{\ep}'(c_{\ep}).
        \nonumber
    \end{equation}
\end{itemize}

From this definition it follows that now the junction in $b_{\ep}$ is of class $C^1$, and hence $\hatv_{\ep}\in H^2((a_{\ep},c_{\ep}))$, and in addition $\hatv_{\ep}-v_{\ep}\in H^2_0((a_{\ep},c_{\ep}))$, and therefore
\begin{equation}
    \RPMF_{\ep}(\beta,g_{\ep},(a_{\ep},c_{\ep}),v_{\ep})\leq
    \RPMF_{\ep}(\beta,g_{\ep},(a_{\ep},c_{\ep}),\hatv_{\ep}).
    \label{est:RPMF-vep-hatvep}
\end{equation}

Now the idea is that the energy of $\hatv_{\ep}$ in the two sub-intervals $(a_\ep,b_\ep)$ and $(b_\ep,c_\ep)$ is close to the energy of $\overline{v}_{\ep}$ in the same intervals, which in turn can be estimated from above by the right-hand side of (\ref{th:RPMF<=}). More precisely, when $\ep$ is small enough we claim that in the interval $(a_{\ep},b_{\ep})$ it turns out that
\begin{equation}
    \RPMF_{\ep}(\beta,g_{\ep},(a_{\ep},b_{\ep}),\overline{v}_{\ep})\leq
    \beta\int_{a_{\ep}}^{b_{\ep}}g_{\ep}(y)^2\,dy+\gamma_7 m_{\ep}
    \label{est:w-ab}
\end{equation}
and
\begin{equation}
    \RPMF_{\ep}(\beta,g_{\ep},(a_{\ep},b_{\ep}),\hatv_{\ep})\leq
    \RPMF_{\ep}(\beta,g_{\ep},(a_{\ep},b_{\ep}),\overline{v}_{\ep})+\gamma_8\omep^2,
    \label{est:w-what-ab}
\end{equation}
while in the interval $(b_{\ep},c_{\ep})$ it turns out that
\begin{equation}
    \RPMF_{\ep}(\beta,g_{\ep},(b_{\ep},c_{\ep}),\overline{v}_{\ep})\leq
    \RPMF_{\ep}(\beta,g_{\ep},(b_{\ep},c_{\ep}),v_{\ep})+\gamma_9 m_{\ep}
    \label{est:w-bc}
\end{equation}
and
\begin{equation}
    \RPMF_{\ep}(\beta,g_{\ep},(b_{\ep},c_{\ep}),\hatv_{\ep})\leq
    \RPMF_{\ep}(\beta,g_{\ep},(b_{\ep},c_{\ep}),\overline{v}_{\ep})+\gamma_{10}\omep^2.
    \label{est:w-what-bc}
\end{equation}

Thanks to (\ref{est:RPMF-vep-hatvep}), these four inequalities are enough to establish (\ref{th:RPMF<=}). Let us prove them.

In the interval $[a_{\ep},b_{\ep}]$ it turns out that $|\overline{v}_{\ep}(y)|=|v_{\ep}(a_{\ep})|\leq m_{\ep}$, and therefore from the uniform boundedness of $g_{\ep}(y)$ if follows that
\begin{equation}
    \RPMF_{\ep}(\beta,g_{\ep},(a_{\ep},b_{\ep}),\overline{v}_{\ep})=
    \beta\int_{a_{\ep}}^{b_{\ep}}(\overline{v}_{\ep}(y)-g_{\ep}(y))^2\,dy
\nonumber
\end{equation}
satisfies (\ref{est:w-ab}). In particular the left-hand side is bounded independently of $\ep$.

At this point in the interval $[a_{\ep},b_{\ep}]$ we can apply Lemma~\ref{lemma:Lip-BCs} to the functions $\hatv:=\hatv_{\ep}$ and $v:=\overline{v}_{\ep}$, whose values coincide for $y\in\{a_{\ep},b_{\ep}\}$. Since
\begin{equation}
    |\hatv_{\ep}'(a_{\ep})-\overline{v}_{\ep}'(a_{\ep})|=|v_{\ep}'(a_{\ep})|
    \qquad\text{and}\qquad
    |\hatv_{\ep}'(b_{\ep})-\overline{v}_{\ep}'(b_{\ep})|=0,
\nonumber
\end{equation}
from (\ref{th:Lip-BCs}) and (\ref{est:vep'-abc}) we conclude that
\begin{equation}
    \RPMF_{\ep}(\beta,g_{\ep},(a_{\ep},b_{\ep}),\hatv_{\ep})\leq 
    \RPMF_{\ep}(\beta,g_{\ep},(a_{\ep},b_{\ep}),\overline{v}_{\ep})+\Gamma_0\ep^{4/3} \gamma_1\omep,
\nonumber
\end{equation}
which implies (\ref{est:w-what-ab}) when $\ep$ is small enough (here it is essential that $4/3\geq 1$).

In order to prove (\ref{est:w-bc}) and (\ref{est:w-what-bc}), from (\ref{th:lim-abc}) we deduce that $H_{\ep}\to 1$, and in particular $H_{\ep}$ is bounded, and
\begin{equation}
    |H_{\ep}-1|=
    \frac{|v_{\ep}(b_{\ep})-v_{\ep}(a_{\ep})|}{v_{\ep}(c_{\ep})-v_{\ep}(b_{\ep})}\leq
    \gamma_{11} m_{\ep}.
\label{est:Hep-1}\end{equation}

Now let us estimate $\RPMF_{\ep}(\beta,g_{\ep},(b_{\ep},c_{\ep}),\overline{v}_{\ep})-\RPMF_{\ep}(\beta,g_{\ep},(b_{\ep},c_{\ep}),v_{\ep})$. As for the term with second order derivatives, we obtain that
\begin{equation}
    \int_{b_{\ep}}^{c_{\ep}}\ep^6\left(\overline{v}_{\ep}''(y)^2-v_{\ep}''(y)^2\right)\,dy=
    (H_{\ep}^2-1)\int_{b_{\ep}}^{c_{\ep}}\ep^6 v_{\ep}''(y)^2\,dy\leq
    \gamma_{12}m_\ep,
    \label{est:w-wep-s}
\end{equation}
where the last inequality follows from (\ref{est:Hep-1}) and (\ref{th:RPM-bounded}). As for the fidelity term, from the uniform boundedness of $v_\ep$ and (\ref{est:Hep-1}) we deduce that
\begin{equation}
    |\overline{v}_{\ep}(y)-v_{\ep}(y)|=
    |(H_{\ep}-1)(v_{\ep}(y)-v_{\ep}(c_{\ep}))|\leq
    \gamma_{13} m_{\ep}
    \qquad
    \forall y\in[b_{\ep},c_{\ep}],
\nonumber
\end{equation}
and therefore, since $g_\ep$ is uniformly bounded as well,
\begin{equation}
    \int_{b_{\ep}}^{c_{\ep}}(\overline{v}_{\ep}(y)-g_{\ep}(y))^2\,dy\leq
    \int_{b_{\ep}}^{c_{\ep}}(v_{\ep}(y)-g_{\ep}(y))^2\,dy+\gamma_{14} m_{\ep}.
    \label{est:w-wep-f}
\end{equation}

Finally, we observe that 
\begin{equation}
    \log\left(1+\delta\frac{\sigma^2}{1+\sigma^2}\right)\leq
    |\delta|\frac{\sigma^2}{1+\sigma^2}\leq
    |\delta|\log\left(1+\sigma^2\right)
    \qquad
    \forall\delta>-1, \quad
    \forall\sigma\in\re,
\nonumber
\end{equation}    
where in the second inequality we exploited the monotonicity for $x\geq 0$ of the function $x\mapsto x-(1+x)\log(1+x)$. Applying this inequality with $\delta:=H_\ep^2-1$ and $\sigma=v_\ep'(y)$, we obtain that the terms with first order derivatives satisfy
\begin{eqnarray}
    \lefteqn{\hspace{-5em}
    \frac{1}{\omep^2}\int_{b_{\ep}}^{c_{\ep}}\log\left(1+\overline{v}_{\ep}'(y)^2\right)\,dy-
    \frac{1}{\omep^2}\int_{b_{\ep}}^{c_{\ep}}\log\left(1+v_{\ep}'(y)^2\right)\,dy}
\nonumber
    \\[0.5ex]
    & = &
    \frac{1}{\omep^2}\int_{b_{\ep}}^{c_{\ep}}\log\left(1+(H_{\ep}^2-1)\frac{v_{\ep}'(y)^2}{1+v_{\ep}'(y)^2}\right)\,dy
\nonumber
    \\[0.5ex]
    & \leq &
   |H_{\ep}^2-1|\frac{1}{\omep^2}\int_{b_{\ep}}^{c_{\ep}}\log\left(1+v_{\ep}'(y)^2\right)\,dy
\nonumber
    \\[0.5ex]
    & \leq &
    \gamma_{15} m_{\ep},
    \label{est:w-wep-pm}
\end{eqnarray}
where again the last inequality follows from (\ref{est:Hep-1}) and (\ref{th:RPM-bounded}).

Summing (\ref{est:w-wep-s}), (\ref{est:w-wep-f}) and (\ref{est:w-wep-pm}) we deduce (\ref{est:w-bc}), which implies in particular that the left-hand side is bounded independently of $\ep$.

Therefore, in the interval $[b_{\ep},c_{\ep}]$ we can apply Lemma~\ref{lemma:Lip-BCs} to the functions $\hatv:=\hatv_{\ep}$ and $v:=\overline{v}_{\ep}$, because their values for $y\in\{b_{\ep},c_{\ep}\}$ coincide. Since
\begin{equation}
    |\hatv_{\ep}'(c_{\ep})-\overline{v}_{\ep}'(c_{\ep})|=
    |H_{\ep}-1|\cdot|v_{\ep}'(c_{\ep})|\leq\gamma_{16}m_{\ep}\cdot\omep
\nonumber
\end{equation}
and
\begin{equation}
    |\hatv_{\ep}'(b_{\ep})-\overline{v}_{\ep}'(b_{\ep})|=
    |H_{\ep} v_{\ep}'(b_{\ep})|\leq\gamma_{17}\omep,
\nonumber
\end{equation}
from (\ref{th:Lip-BCs}) we deduce (\ref{est:w-what-bc}) as in the previous case.


\subsection{Boundedness of higher order norms and compactness}

Let $[a,b]\subseteq(-H,H)$ be any interval. In this final part of the proof we show that the family $\{w_\ep\}$ defined by (\ref{defn:wep-hor}) is relatively compact in $C^1([a,b])$. Moreover, we characterize its possible limit points as in (\ref{th:hor-structure}).

To this end, we consider real numbers
\begin{equation}
    -H<a''<a'<a<b<b'<b''<H.
\nonumber
\end{equation}

From the results of section~\ref{sec:horiz-1} applied in the interval $[a'',b'']$, we know that the family $\{w_\ep\}$ is bounded in $C^0([a'',b''])$. In order to obtain a bound on derivatives, we apply the mean value theorem in the intervals $[a'',a']$ and $[b',b'']$, and we deduce that there exist two families $\{a_\ep\}\subseteq[a'',a']$ and $\{b_\ep\}\subseteq[b',b'']$, and a constant $M_1$ such that
\begin{equation}
    |w_\ep(a_\ep)|+|w_\ep(b_\ep)|+|w_\ep'(a_\ep)|+|w_\ep'(b_\ep)|\leq M_1
    \label{hp:BC-wep}
\end{equation}
for $\ep$ small enough.

Now the idea is the following. We observe that
\begin{multline*}
\RPMF_{\ep}(\beta,g_{\ep},(a_{\ep},b_{\ep}),v_{\ep})
\\
=\RPM_{\ep}((a_{\ep},b_{\ep}),\omep^2 w_{\ep}) 
 +\beta\int_{a_{\ep}}^{b_{\ep}}
\left(\omep^{4}w_{\ep}(y)^{2}-2\omep^{2}g_{\ep}(y)w_{\ep}(y)+g_{\ep}(y)^{2}\right)dy,
\end{multline*}
so that we can consider $w_{\ep}$ as a minimizer to the right-hand side subject to its own boundary conditions. Since the integral of $g_{\ep}(y)^{2}$ plays no role in the minimization process, we can neglect it and divide by $\omep^{2}$. This leads to consider the family of rescaled Perona-Malik functionals for horizontal regions, defined as
\begin{eqnarray*}
\RPMH_{\ep}(\Omega,w)&:=&
\int_{\Omega}\left(\ep ^6\omep^2 w''(y)^2+\frac{1}{\omep^4}\log\left(1+\omep^4 w'(y)^2 \right)\right)dy
\\
& &
\mbox{}+\beta\int_{\Omega}\left(\omep^{2}w(y)^{2}-2g_{\ep}(y)w(y)\right)dy.
\end{eqnarray*}

Now from statement~(2) of Theorem~\ref{thmbibl:add-on} we know that $g_{\ep}(y)\to f'(x_0)y$ uniformly on bounded sets ($a_0=0$ because in this case the origin is the midpoint of the horizontal part of a step), and hence it is reasonable to expect that $\RPMH_{\ep}$ converges to the quadratic functional
\begin{equation*}
\int_{\Omega}\left[w'(y)^{2}-2\beta f'(x_0)y\cdot w(y)\right]dy,
\end{equation*}
and $w_{\ep}$ converges to some minimizer $w_0$ of the limit problem. Recalling that $w_{\ep}(0)=0$ for every $\ep\in(0,1)$, we conclude that $w_0(0)=0$, and hence the possible limits are only those given by the right-hand side of (\ref{th:hor-structure}).

This is actually what happens, and the proof follows from Proposition~\ref{lemma:ECV2} applied to the family of functionals $\FF_\ep:=\RPMH_\ep$ with
\begin{equation}
    \tau_\ep^2:=\ep^6\omep^2,
    \quad
    \eta_\ep^2:=\beta\omep^2,
    \quad
    \varphi_{\ep}(\sigma):=\frac{1}{\omep^4}\log\left(1+\omep^4 \sigma^2\right),
    \quad
    h_\ep(y)=\beta g_\ep(y).
\nonumber
\end{equation}

All the hypotheses are fulfilled, but for the fact that $\varphi_\ep$ does not satisfy the uniform convexity assumption (\ref{hp:phin-convex}), and is actually not even convex. To address this issue, we introduce the quantity
\begin{equation}
    D_{\ep}:=\frac{1}{6\omep^2},
\nonumber
\end{equation}
and we observe that
\begin{equation}
    \varphi_{\ep}''(\sigma)\geq 1
    \qquad
    \forall\sigma\in[-2D_{\ep},2D_{\ep}].
\nonumber
\end{equation}

Therefore, if we prove that
\begin{equation}
    |w_{\ep}'(y)|\leq 2D_{\ep}
    \qquad
    \forall y\in[a_{\ep},b_{\ep}]
    \label{th:wn-convex}    
\end{equation}
for every $\ep$ small enough, then we can pretend that the function $\varphi_{\ep}$ is convex, and we can finally apply Proposition~\ref{lemma:ECV2}.

In order to prove (\ref{th:wn-convex}), we assume by contradiction that it is false for some $\ep>0$, and as a consequence the open set
\begin{equation}
\mathcal{A}_{\ep}:=\left\{y\in(a_{\ep},b_{\ep}):|w_{\ep}'(y)|>D_{\ep}\right\}
\label{defn:Aep-final}
\end{equation}
is nonempty. In this case we introduce the modified function $\hatw_{\ep}\in H^{2}((a_{\ep},b_{\ep}))$ that satisfies $\hatw_{\ep}(a_{\ep})=w_{\ep}(a_{\ep})$, and whose derivative is given for every $y\in(a_\ep,b_\ep)$ by
\begin{equation*}
\hatw_{\ep}'(y):=\begin{cases}
-D_{\ep} & \text{if }w_{\ep}'(y)<-D_{\ep}, \\
w_{\ep}'(y)\qquad & \text{if }|w_{\ep}'(y)|\leq D_{\ep}, \\
D_{\ep} & \text{if }w_{\ep}'(y)>D_{\ep}.
\end{cases}
\end{equation*}

For the sake of shortness, we write the functional (\ref{defn:Fn}) in the form
\begin{equation}
    \RPMH_{\ep}(\Omega,w)=\FF_{\ep}^s(\Omega,w)+\FF_{\ep}^{pm}(\Omega,w)+\FF_{\ep}^f(\Omega,w),
\nonumber
\end{equation}
where $\FF_{\ep}^s$ is the term with the second order derivative (the singular perturbation term), $\FF_{\ep}^{pm}$ is the term with the logarithm (the Perona-Malik term), and $\FF_{\ep}^{f}$ is the remaining part (the fidelity term). With this notation we observe that
\begin{equation}
    \FF_{\ep}^s((a_{\ep},b_{\ep})\setminus \mathcal{A}_{\ep},w_{\ep})=
    \FF_{\ep}^s((a_{\ep},b_{\ep})\setminus \mathcal{A}_{\ep},\hatw_{\ep})=
    \FF_{\ep}^s((a_{\ep},b_{\ep}),\hatw_{\ep}),
\nonumber
\end{equation}
and
\begin{equation}
    \FF_{\ep}^{pm}((a_{\ep},b_{\ep}),w_{\ep})\geq
    \FF_{\ep}^{pm}((a_{\ep},b_{\ep}),\hatw_{\ep}).
\nonumber
\end{equation}

In addition, assuming always that (\ref{th:wn-convex}) is false, we claim that there exists three real numbers $M_2$, $M_3$, $M_4$, with $M_2>0$, such that the following three inequalities hold true whenever $\ep$ is small enough:
\begin{gather}
    \FF_{\ep}^s(\mathcal{A}_{\ep},w_{\ep})\geq\dfrac{M_2}{|\log\ep|^3},
    \label{claim:Fn-1}
    \\[0.5ex]
    \FF_{\ep}^f((a_{\ep},b_{\ep}),w_{\ep})\geq \FF_{\ep}^f((a_{\ep},b_{\ep}),\hatw_{\ep})-M_3\ep^2|\log\ep|^{5/2},
    \label{claim:Fn-2}
    \\[1ex]
    \RPMH_{\ep}((a_{\ep},b_{\ep}),\hatw_{\ep})\geq \RPMH_{\ep}((a_{\ep},b_{\ep}),w_{\ep})-M_4\ep^2|\log\ep|^{5/2}.
    \label{claim:Fn-3}
\end{gather}

If we prove these claims, then (\ref{th:wn-convex}) follows. Indeed, from all these estimates we obtain the chain of inequalities
\begin{eqnarray*}
\RPMH_{\ep}((a_{\ep},b_{\ep}),w_{\ep}) & = & 
\FF_{\ep}^s(\mathcal{A}_{\ep},w_{\ep})+\FF_{\ep}^s((a_{\ep},b_{\ep})\setminus \mathcal{A}_{\ep},w_{\ep})+\FF_{\ep}^{pm}((a_{\ep},b_{\ep}),w_{\ep})
\\[0.5ex]
& & 
\mbox{}+\FF_{\ep}^f((a_{\ep},b_{\ep}),w_{\ep})
\\[0.5ex]
& \geq &
\dfrac{M_2}{|\log\ep|^3}+\FF_{\ep}^s((a_{\ep},b_{\ep}),\hatw_{\ep})+\FF_{\ep}^{pm}((a_{\ep},b_{\ep}),\hatw_{\ep})
\\[0.5ex]
& & 
\mbox{}+\FF_{\ep}^f((a_{\ep},b_{\ep}),\hatw_{\ep})-M_3\ep^2|\log\ep|^{5/2} 
\\[0.5ex]
& = &
\dfrac{M_2}{|\log\ep|^3}+\RPMH_{\ep}((a_{\ep},b_{\ep}),\hatw_{\ep})-M_3\ep^2|\log\ep|^{5/2} 
\\[0.5ex]
& \geq &
\RPMH_{\ep}((a_{\ep},b_{\ep}),w_{\ep})+\dfrac{M_2}{|\log\ep|^3}-(M_3+M_4)\ep^2|\log\ep|^{5/2} 
\end{eqnarray*}

If we factor out the common term $\RPMH_{\ep}((a_{\ep},b_{\ep}),w_{\ep})$, and we multiply by $|\log\ep|^3$, we obtain an inequality that is impossible if $\ep$ is small enough, which gives the required contradiction. 

\paragraph{\textmd{\textit{Proof of (\ref{claim:Fn-1})}}}

To begin with, we show that there exist two real numbers $ M_5$ and $ M_6$ such that
\begin{equation}
    \RPMH_{\ep}((a_{\ep},b_{\ep}),w_{\ep})\leq M_5,
    \qquad
    \FF_{\ep}^s((a_{\ep},b_{\ep}),w_{\ep})+\FF_{\ep}^{pm}((a_{\ep},b_{\ep}),w_{\ep})\leq M_6.
    \label{th:bound-HRPM}
\end{equation}

In order to prove the first inequality, it is enough to compare $\RPMH_{\ep}((a_{\ep},b_{\ep}),w_{\ep})$ with the value in any other admissible competitor, for example the unique cubic polynomial with the same boundary conditions as $w_{\ep}$. The boundary conditions of $w_{\ep}$ are bounded independently of $\ep$ because of (\ref{hp:BC-wep}), hence also the coefficients of the cubic polynomial are bounded independently of $\ep$, and as a consequence the values of $\RPMH_{\ep}$, computed in this cubic polynomial, are bounded independently of $\ep$. This proves the first inequality. At this point the second inequality follows from the first one because $\FF_{\ep}^f((a_{\ep},b_{\ep}),w_{\ep})$ is bounded independently of $\ep$ due to the uniform bounds on $w_{\ep}$ and $g_{\ep}$.

Now we apply Lemma~\ref{lemma:stima-ABGabc} with
\begin{equation}
    \lambda=\ep^6\omep^2,
    \qquad
    \mu:=\frac{1}{\omep^4},
    \qquad
    \nu:=\omep^4,
    \qquad
    D=M:=D_{\ep}
    \label{choices:lemma}
\end{equation}
and from (\ref{th:est-AM}) we obtain that there exists a constant $ M_7$ such that
\begin{equation}
    |\mathcal{A}_{\ep}|\leq
    \FF_\ep^{pm}((a_\ep,b_\ep),w_\ep)\cdot\frac{\omep^4}{\log(1+\omep^4 D_\ep^2)}
    \leq M_7\omep^4
\nonumber
\end{equation}
when $\ep$ is small enough.

Since we are assuming that (\ref{th:wn-convex}) is false, we know that there exists $c_{\ep}\in(a_{\ep},b_{\ep})$ such that $|w_{\ep}'(c_{\ep})|=2D_{\ep}$. Let us consider the smallest real number $d_{\ep}\in(c_{\ep},b_{\ep})$ such that $|w_{\ep}'(d_{\ep})|=D_{\ep}$, which exists because $|w_{\ep}'(b_{\ep})|<D_{\ep}$ when $\ep$ is small enough. Now we observe that $(c_{\ep},d_{\ep})\subseteq \mathcal{A}_{\ep}$, and as a consequence
\begin{equation}
    d_{\ep}-c_{\ep}\leq|\mathcal{A}_{\ep}|\leq M_7\omep^4. 
\nonumber
\end{equation}

It follows that
\begin{eqnarray*}
    \FF_{\ep}^s(\mathcal{A}_{\ep},w_{\ep}) & \geq & 
    \ep^{6}\omep^2\int_{c_{\ep}}^{d_{\ep}}w_{\ep}''(y)^2\,dy
    \\
    & \geq &
    \ep^{6}\omep^2 \frac{1}{d_{\ep}-c_{\ep}}\left(\int_{c_{\ep}}^{d_{\ep}}w_{\ep}''(y)\,dy\right)^{2}
    \\
    & = &
    \ep^{6}\omep^2 \frac{1}{d_{\ep}-c_{\ep}}\left(w_{\ep}'(d_{\ep})-w_{\ep}'(c_{\ep})\right)^{2}
    \\
    & \geq &
    \ep^{6}\omep^2 \frac{1}{ M_7\omep^4}D_{\ep}^2,    
\end{eqnarray*}
which proves (\ref{claim:Fn-1}).

\paragraph{\textmd{\textit{Proof of (\ref{claim:Fn-2})}}}

We claim that there exists a constant $ M_8$ such that
\begin{equation}
    \int_{a_{\ep}}^{b_{\ep}}|w_{\ep}'(y)-\hatw_{\ep}'(y)|\,dy\leq
     M_8 \ep^2|\log\ep|^{5/2}
    \label{th:wn'-hatwn'}
\end{equation}
when $\ep$ is small enough. Since the two functions coincide in $a_\ep$, from this inequality it follows that
\begin{equation}
    |w_{\ep}(y)-\hatw_{\ep}(y)|\leq
     M_8\ep^2|\log\ep|^{5/2}
    \qquad
    \forall y\in[a_\ep,b_\ep],
    \label{est:wn-hatwn}
\end{equation}
which in turn implies (\ref{claim:Fn-2}) because of the uniform bound on $b_{\ep}-a_{\ep}$, $w_\ep(y)$ and $g_{\ep}(y)$.

In order to prove (\ref{th:wn'-hatwn'}) we observe that
\begin{equation}
    \int_{a_{\ep}}^{b_{\ep}}|w_{\ep}'(y)-\hatw_{\ep}'(y)|\,dy=
    \int_{\mathcal{A}_{\ep}}|w_{\ep}'(y)-\hatw_{\ep}'(y)|\,dy=
    \int_{\mathcal{A}_{\ep}}|w_{\ep}'(y)|\,dy-|\mathcal{A}_\ep|\cdot D_\ep.
\nonumber
\end{equation}

Applying again Lemma~\ref{lemma:stima-ABGabc} with the choices (\ref{choices:lemma}), from (\ref{th:high-deriv}) we obtain that the latter is less than or equal to
\begin{equation}
    \left[\FF_\ep^s((a_\ep,b_\ep),w_\ep)+\FF_\ep^{pm}((a_\ep,b_\ep),w_\ep)\strut\right]^2\cdot\frac{3}{32}\left[\ep^6\omep^2\cdot\frac{1}{\omep^{12}}\cdot\log\left(1+\omep^4 D_\ep^2\right)\right]^{-1/2},
\nonumber
\end{equation}
which implies (\ref{th:wn'-hatwn'}) because of the second inequality in (\ref{th:bound-HRPM}).

\paragraph{\textmd{\textit{Proof of (\ref{claim:Fn-3})}}}

To begin with, we prove that there exists a constant $ M_9$ such that
\begin{equation}    
\int_{a_{\ep}}^{b_{\ep}}|\hatw_{\ep}'(y)|\,dy\leq
\int_{a_{\ep}}^{b_{\ep}}|w_{\ep}'(y)|\,dy\leq
 M_9
\label{est:hatwn'}
\end{equation}
when $\ep$ is small enough.

To this end, in addition to (\ref{defn:Aep-final}) we introduce the sets
\begin{equation}
    \mathcal{C}_\ep:=\left\{y\in(a_{\ep},b_{\ep}):|w_{\ep}'(y)|<D_{\ep}\right\},
    \qquad
    \mathcal{B}_\ep:=\left\{y\in(a_{\ep},b_{\ep}):|w_{\ep}'(y)|=D_{\ep}\right\},
\nonumber
\end{equation}
and we write the total variation of $w_\ep$ as
\begin{equation}
    \int_{\mathcal{C}_\ep}|w_{\ep}'(y)|\,dy+
    \int_{\mathcal{B}_\ep}|w_{\ep}'(y)|\,dy+
    \left(\int_{\mathcal{A}_\ep}|w_{\ep}'(y)|\,dy-|\mathcal{A}_\ep|\cdot D_\ep\right)+
    |\mathcal{A}_\ep|\cdot D_\ep
\nonumber
\end{equation}

At this point we apply once again Lemma~\ref{lemma:stima-ABGabc} with the choices (\ref{choices:lemma}), and thanks to the second inequality in (\ref{th:bound-HRPM}) we obtain that the four terms are bounded as required.

Now we choose a function $\psi\in C^\infty(\re)$ such that $\psi(y)=0$ for every $y\leq a$ and $\psi(y)=1$ for every $y\geq b$. For every  $\ep\in(0,1)$ we consider the constant
\begin{equation}
    \Delta_{\ep}:=w_{\ep}(b_{\ep})-\hatw_{\ep}(b_{\ep}),
\nonumber
\end{equation}
and the function
\begin{equation}
    w_{*,\ep}(y):=\hatw_{\ep}(y)+\Delta_{\ep}\cdot\psi(y)
    \qquad
    \forall y\in[a_{\ep},b_{\ep}].
    \label{defn:w*n}
\end{equation}

We observe that $w_\ep-w_{*,\ep}\in H^2_0((a_\ep,b_\ep))$, and from (\ref{est:wn-hatwn}) we deduce that
\begin{equation}
    |\Delta_{\ep}|\leq
     M_8\ep^2|\log\ep|^{5/2}.
    \label{est:Cn}
\end{equation}

Since $w_{\ep}$ is a minimum point of $\RPMH_{\ep}((a_{\ep},b_{\ep}),w)$ with respect to its boundary conditions, this implies that
\begin{equation}
    \RPMH_{\ep}((a_{\ep},b_{\ep}),w_{\ep})\leq \RPMH_{\ep}((a_{\ep},b_{\ep}),w_{*,\ep}).
\nonumber
\end{equation}

Therefore, in order to prove (\ref{claim:Fn-3}) it is enough to show that
\begin{equation}
    \RPMH_{\ep}((a_{\ep},b_{\ep}),w_{*,\ep})\leq
    \RPMH_{\ep}((a_{\ep},b_{\ep}),\hatw_{\ep})+M_4\ep^2|\log\ep|^{5/2}.
\nonumber
\end{equation}

We show more precisely that an inequality of this type is true separately for each of the three functionals that sum up to $\RPMH_{\ep}$.

In the case of $\FF_{\ep}^s$ and $\FF_{\ep}^f$ the estimate follows almost immediately from (\ref{defn:w*n}) and (\ref{est:Cn}). In the case of $\FF_{\ep}^{pm}$ we start from the inequality
\begin{equation*}
    \quad
    \frac{1}{\nu}\log\left(1+\nu(\sigma+\tau)^2\right)-\frac{1}{\nu}\log\left(1+\nu\sigma^2\right)= 
    \frac{1}{\nu}\log\left(1+\nu\frac{2\sigma\tau+\tau^2}{1+\nu\sigma^2}\right)
    \leq
    2|\sigma|\cdot|\tau|+\tau^2,
    \quad
\end{equation*}
which is true for every triple of real numbers $(\nu,\sigma,\tau)$ with $\nu>0$. Applying it with 
\begin{equation}
    \nu:=\omep^4,
    \qquad\quad
    \sigma=\hatw_{\ep}'(y),
    \qquad\quad
    \tau:=\Delta_{\ep}\cdot\psi'(y),
\nonumber
\end{equation}
we deduce that
\begin{multline*}
    \qquad
    \FF_{\ep}^{pm}((a_{\ep},b_{\ep}),w_{*,\ep})\leq
    \FF_{\ep}^{pm}((a_{\ep},b_{\ep}),\hatw_{\ep})
    \\
    \mbox{}+2|\Delta_{\ep}|\int_{a_{\ep}}^{b_{\ep}}|\hatw_{\ep}'(y)|\cdot|\psi'(y)|\,dy+
    \Delta_{\ep}^2\int_{a_{\ep}}^{b_{\ep}}\psi'(y)^2\,dy,
    \qquad
\end{multline*}
so that the conclusion follows from (\ref{est:hatwn'}) and (\ref{est:Cn}).
\qed


\subsubsection*{\centering Acknowledgments}

The authors are members of the \selectlanguage{italian} ``Gruppo Nazionale per l'Analisi Matematica, la Probabilità e le loro Applicazioni'' (GNAMPA) of the ``Istituto Nazionale di Alta Matematica'' (INdAM).

The authors acknowledge the MIUR Excellence Department Project awarded to the Department of Mathematics, University of Pisa, CUP I57G22000700001.

\selectlanguage{english}


\begin{thebibliography}{10}
\providecommand{\url}[1]{\texttt{#1}}
\providecommand{\urlprefix}{URL }
\providecommand{\selectlanguage}[1]{\relax}
\providecommand{\eprint}[2][]{\url{#2}}

\bibitem{ABG}
\textsc{R.~Alicandro}, \textsc{A.~Braides}, \textsc{M.~S. Gelli}.
\newblock Free-discontinuity problems generated by singular perturbation.
\newblock \emph{Proc. Roy. Soc. Edinburgh Sect. A} \textbf{128} (1998), no.~6,
  1115--1129.

\bibitem{AFP}
\textsc{L.~Ambrosio}, \textsc{N.~Fusco}, \textsc{D.~Pallara}.
\newblock \emph{Functions of bounded variation and free discontinuity
  problems}.
\newblock Oxford Mathematical Monographs. The Clarendon Press, Oxford
  University Press, New York, 2000.

\bibitem{2008-TAMS-BF}
\textsc{G.~Bellettini}, \textsc{G.~Fusco}.
\newblock The {$\Gamma$}-limit and the related gradient flow for singular
  perturbation functionals of {P}erona-{M}alik type.
\newblock \emph{Trans. Amer. Math. Soc.} \textbf{360} (2008), no.~9,
  4929--4987.

\bibitem{2006-DCDS-BelFusGug}
\textsc{G.~Bellettini}, \textsc{G.~Fusco}, \textsc{N.~Guglielmi}.
\newblock A concept of solution and numerical experiments for forward-backward
  diffusion equations.
\newblock \emph{Discrete Contin. Dyn. Syst.} \textbf{16} (2006), no.~4,
  783--842.

\bibitem{2006-SIAM-BNP}
\textsc{G.~Bellettini}, \textsc{M.~Novaga}, \textsc{E.~Paolini}.
\newblock Global solutions to the gradient flow equation of a nonconvex
  functional.
\newblock \emph{SIAM J. Math. Anal.} \textbf{37} (2006), no.~5, 1657--1687.

\bibitem{2019-SIAM-BerGiaTes}
\textsc{M.~Bertsch}, \textsc{L.~Giacomelli}, \textsc{A.~Tesei}.
\newblock Measure-valued solutions to a nonlinear fourth-order regularization
  of forward-backward parabolic equations.
\newblock \emph{SIAM J. Math. Anal.} \textbf{51} (2019), no.~1, 374--402.

\bibitem{2020-JDE-BerSmaTes}
\textsc{M.~Bertsch}, \textsc{F.~Smarrazzo}, \textsc{A.~Tesei}.
\newblock On a class of forward-backward parabolic equations: formation of
  singularities.
\newblock \emph{J. Differential Equations} \textbf{269} (2020), no.~9,
  6656--6698.

\bibitem{1996-Duke-DeGiorgi}
\textsc{E.~De~Giorgi}.
\newblock Conjectures concerning some evolution problems.
\newblock \emph{Duke Math. J.} \textbf{81} (1996), no.~2, 255--268.
\newblock A celebration of John F. Nash, Jr.

\bibitem{2001-CPAM-Esedoglu}
\textsc{S.~Esedo\={g}lu}.
\newblock An analysis of the {P}erona-{M}alik scheme.
\newblock \emph{Comm. Pure Appl. Math.} \textbf{54} (2001), no.~12, 1442--1487.

\bibitem{GG:grad-est}
\textsc{M.~Ghisi}, \textsc{M.~Gobbino}.
\newblock Gradient estimates for the {P}erona-{M}alik equation.
\newblock \emph{Math. Ann.} \textbf{337} (2007), no.~3, 557--590.

\bibitem{FastPM-CdV}
\textsc{M.~Gobbino}, \textsc{N.~Picenni}.
\newblock A quantitative variational analysis of the staircasing phenomenon for
  a second order regularization of the {P}erona-{M}alik functional.
\newblock \emph{Trans. Amer. Math. Soc.} \textbf{376} (2023), 5307--5375.

\bibitem{2023-PM-TV}
\textsc{M.~Gobbino}, \textsc{N.~Picenni}.
\newblock Monotonicity properties of limits of solutions to the semidiscrete
  scheme for a class of {P}erona-{M}alik type equations.
\newblock \emph{SIAM J. Math. Anal.} \textbf{56} (2024), no.~2, 2034--2062.

\bibitem{Kichenassamy}
\textsc{S.~Kichenassamy}.
\newblock The {P}erona-{M}alik paradox.
\newblock \emph{SIAM J. Appl. Math.} \textbf{57} (1997), no.~5, 1328--1342.

\bibitem{2018-Nonlinearity-KimYan}
\textsc{S.~Kim}, \textsc{B.~Yan}.
\newblock On asymptotic behavior and energy distribution for some
  one-dimensional non-parabolic diffusion problems.
\newblock \emph{Nonlinearity} \textbf{31} (2018), no.~6, 2756--2808.

\bibitem{1973-JLL}
\textsc{J.-L. Lions}.
\newblock \emph{Perturbations singuli\`eres dans les probl\`emes aux limites et
  en contr\^{o}le optimal}, \emph{Lecture Notes in Mathematics}, volume Vol.
  323.
\newblock Springer-Verlag, Berlin-New York, 1973.

\bibitem{PeronaMalik}
\textsc{P.~Perona}, \textsc{J.~Malik}.
\newblock Scale-space and edge detection using anisotropic diffusion.
\newblock \emph{IEEE Transactions on Pattern Analysis and Machine Intelligence}
  \textbf{12} (1990), no.~7, 629--639.

\bibitem{fastpm-discreto}
\textsc{N.~Picenni}.
\newblock Staircasing effect for minimizers of the one-dimensional discrete
  {P}erona-{M}alik functional.
\newblock \emph{ESAIM Control Optim. Calc. Var.}  (2024).
\newblock To appear, {h}ttps://doi.org/10.1051/cocv/2024035.

\bibitem{Solci}
\textsc{M.~Solci}.
\newblock Free-discontinuity problems generated by higher-order singular
  perturbations, 2024.
\newblock ArXiv:2402.10656.

\end{thebibliography}

\label{NumeroPagine}

\end{document}